\documentclass[12pt, reqno]{amsart}

\usepackage{amssymb,amsmath,amsthm,amsfonts}
\usepackage{setspace}
\usepackage{enumerate}
\usepackage{amsaddr}
\usepackage{latexsym}
\usepackage{color}
\usepackage{hyperref}
\hypersetup{
  colorlinks   = true,    
  urlcolor     = blue,    
  linkcolor    = blue,    
  citecolor    = red      
}
\usepackage{graphicx}

\newtheorem{theorem}{Theorem}[section]
\newtheorem{lemma}[theorem]{Lemma}
\newtheorem{proposition}[theorem]{Proposition}
\newtheorem{prop}[theorem]{Proposition}
\newtheorem{cor}[theorem]{Corollary}
\newtheorem{definition}[theorem]{Definition}
\newtheorem{rmk}[theorem]{Remark}

\def\R{{\mathbb R}}

\def\wh{\widehat}
\def\la{\langle}
\def\ra{\rangle}
\def\del{\partial}
\def\tu{\widetilde{u}}
\def\tv{\widetilde{v}}

\newcommand{\norm}[1]{\left \Vert #1 \right \Vert}
\newcommand{\jbrac}[1] {\langle #1 \rangle}
\newcommand{\lsim}{\lesssim}
\newcommand{\gsim}{\gtrsim}

\def\R{{\mathbb R}}
\def\T{{\mathbb T}}

\allowdisplaybreaks 

\begin{document}

\title[Majda-Biello on the half line]{Well-posedness for the Majda-Biello System on the Half Line}

\author{Ellis, Matthew}
\address{Department of Mathematics \\
University of Illinois \\
Urbana, IL 61801, U.S.A.}
\email{ellis23@illinois.edu}
 
\date{}

\begin{abstract}
We study the initial-boundary value problem for the Majda-Biello system posed on the right half line. We prove local well-posedness on the half line, matching the local theory on the real line established by Oh \cite{oh}. The approach combines the Laplace transform method of Bona-Sun-Zhang \cite{bona} with adapted estimates from the work of Colliander and Kenig on the KdV half line initial-boundary value problem \cite{kenig}. 
\end{abstract}

\maketitle

\section{Introduction}

The purpose of this paper is to study the well-posedness and related properties for the Majda-Biello system as an initial-boundary value problem (IBVP) on the half line $\R^+ = (0,\infty)$
\begin{equation} \label{majdaIBVP}
\begin{cases}
u_t + u_{xxx} + vv_x = 0 \\
v_t + \alpha v_{xxx} +  (uv)_x = 0 \qquad x, t \in \R^+ \\
(u,v)|_{t=0} = (u_0, v_0) \in H^s(\R^+), \qquad (u,v)|_{x=0} = (f,g) \in H^{\frac{s+1}{3}}(\R^+).
\end{cases}
\end{equation}

The corresponding system in the periodic setting was originally proposed by Majda-Biello in \cite{majda} as an asymptotic model for the nonlinear interactions of atmospheric waves. Rossby waves, also known as planetary waves, are long wavelength dispersive waves which have important effects on weather patterns and ocean currents. Here $v$ represents an equatorial Rossby wave, and $u$ represents a barotropic Rossby wave with significant mid-latitude projections. We will take the coupling parameter $\alpha$ to be in $(0,1)$. This is the region where the local theory differs most significantly from that for the KdV equation, and it seems to be the most relevant physically. Indeed, Majda and Biello obtained estimates of $0.899$, $0.960$, and $0.980$ for $\alpha$ in the  cases they considered \cite{majda}. 

While the periodic setting ($x \in \T$) is the most natural for studying atmospheric waves, the non-periodic setting ($x \in \R$) is also of interest for its applications to oceanic waves and for theoretical reasons as a generalization of the widely-studied KdV equation. The half line is also a natural setting to consider for wave behavior, corresponding to waves generated at one end and allowed to propagate freely. Well-posedness is more difficult on the half line, and some care must be taken to define the function spaces $H^s(\R^+)$, compatibility between initial and boundary data, and the notion of uniqueness of mild solutions. Before discussing these issues and stating our results, we review the local theory on the full line ($\R$) for the KdV equation and for the Majda-Biello system.

In order to study low regularity solutions to the KdV initial value problem (IVP),
\begin{equation} \label{kdvIVP}
\begin{cases}
u_t + u_{xxx} + uu_x = 0 \quad x \in \R\\
u(x,0) = g(x) \in H^s(\R)
\end{cases}
\end{equation}
Bourgain introduced the restricted norm space, $X^{s,b}$, defined by 
\begin{equation} \label{xsb}
\|u\|_{X^{s,b}(\R \times \R)} = \|\la\xi\ra^s \la \tau-\xi^3\ra^b \widehat{u}(\xi,\tau)\|_{L^2_{\xi,\tau}}
\end{equation}
where $\la \xi \ra = 1 + |\xi|$ and $\widehat{u}(\xi,\tau)$ denotes the space-time Fourier transform. This space makes use of the idea that the space-time Fourier transform of a solution to the linear equation is supported on the curve $\{\tau=\xi^3\}$. By a contraction argument in this space, Bourgain proved in \cite{bourgain} that \eqref{kdvIVP} is locally well-posed for all $s\geq 0$. The $L^2$ conservation law ensures that the local solutions are global in time, so this result implies global well-posedness (GWP) for the KdV on $\R$ for all $s\geq 0$. The proof relies on a key bilinear estimate
\begin{equation} \label{bil_kdv}
\| \partial_x(uv) \|_{X^{s,b-1}(\R)} \lsim \|u\|_{X^{s,b}} \|v\|_{X^{s,b}} \quad \text{for } s\geq 0, \,  b>\frac12.
\end{equation}

Kenig-Ponce-Vega showed that \eqref{bil_kdv} in fact holds for $s>-\frac34$, thereby establishing local well-posedness (LWP) for the KdV equation on $\R$ for $s>-\frac34$, \cite{kpv}. Moreover, this result almost sharp as the KdV equation is ill-posed for $s<-\frac34$, in the sense that the data-to-solution map $u_0 \to u$ from $H^s(\R)$ to $H^s(\R)$ is not $C^2$ \cite{tzv}. GWP for $s>-\frac34$ was obtained by Colliander-Keel-Staffilani-Takaoka-Tao \cite{colliander} via the I-method. Although the bilinear estimate \eqref{bil_kdv} fails at the endpoint $s = -\frac34$, the KdV equation was shown to be locally \cite{christ} and globally \cite{guo} well-posed for $s = -\frac34$ by making use of the (modified) Miura transform.

Returning to the Majda-Biello system, we consider the initial value problem
\begin{equation} \label{majdaIVP}
\begin{cases}
u_t + u_{xxx} + vv_x = 0  \\
v_t + \alpha v_{xxx} +  (uv)_x = 0 \quad x \in \R \\
(u,v)|_{t=0} = (u_0, v_0) \quad u_0,v_0 \in H^s(\R)
 \end{cases}
\end{equation}
which was studied by Oh in both the periodic and non-periodic settings \cite{oh}. Compaan showed that the periodic problem exhibits smoothing and studied the existence of global attractors \cite{compaan}. The results on $\R$ are more relevant to the half-line problem however, so we focus on these here. We remark that, unlike the KdV equation on $\R$, the system \eqref{majdaIVP} is not completely integrable, even for $\alpha = 1$. The following quantities are conserved for this system 
\begin{equation} \label{conserved}
\int u\, dx \qquad \int v\, dx \qquad \int u^2+v^2 \, dx \qquad \frac12 \int u_x^2 + \alpha v_x^2 -uv^2 \, dx
\end{equation}
corresponding to masses, energy, and the Hamiltonian. It has been shown that there are no higher conservation laws \cite{vodova}. In contrast, the KdV has infinitely many conservation laws. Powerful methods like the Miura transform and inverse-scattering techniques are not available for \eqref{majdaIVP}. The system  does have scaling, with the same critical regularity as the KdV equation, $s = -\frac32$.

Oh showed that the regularity required for well-posedness depends on the coupling parameter $\alpha$. When $\alpha = 1$, for example, the techniques from the KdV equation carry over, and LWP holds for $s >-\frac34$, just as for the KdV. A key ingredient in establishing the necessary bilinear estimates analogous to \eqref{bil_kdv} is the KdV algebraic frequency relation
\begin{equation} \label{alg_kdv}
\xi^3 - \xi_1^3 -\xi_2^3 = 3 \xi \xi_1 \xi_2 \qquad \text{for } \xi = \xi_1+\xi_2
\end{equation}
which is used to compensate for the derivative in the nonlinearity. The main obstacles in proving the bilinear estimates are thus resonant cases where $\xi \xi_1 \xi_2$ is small.

When $\alpha \neq 1$, the linear semigroups in the two equations are no longer identical, and the space-time Fourier transform of the solutions to the linear equations are now supported on distinct curves, $\{\tau = \xi^3\}$ and $\{\tau =\alpha\xi^3\}$. The relation \eqref{alg_kdv} no longer applies, and the resonant frequency interactions may be more complicated, making the bilinear estimates more difficult to establish. For $\alpha <0$ or $\alpha >4$, it turns out that the resonance equations have no real solutions, and thus LWP holds again for $s>-\frac34$. However, for $0<\alpha<1$, Oh proved LWP for \eqref{majdaIVP} when $s\geq 0$ and that this is sharp if we require the data-to-solution map to be $C^2$.

We turn now to the half line setting, $\R^+ = (0,\infty)$. In \cite{kenig}, Colliander-Kenig consider the KdV initial-boundary value problem (IBVP):

\begin{equation} \label{kdvIBVP}
\begin{cases}
u_t + u_{xxx} + uu_x = 0 \qquad x,t \in \R^+\\
u|_{t=0} = u_0(x) \in H^s(\R^+), \qquad u|_{x=0} = f \in H^{\frac{s+1}{3}}(\R^+)
\end{cases}
\end{equation}

where the $H^s(\R^+)$ norm is defined as:
\begin{equation} \label{R+norm}
\|f\|_{H^s(\R^+)} = \inf \{\|\tilde{f} \|_{H^s(\R)} : \tilde{f}(x) = f(x)\text{ for all }x>0 \}.
\end{equation}

Colliander-Kenig show the existence of solutions to \eqref{kdvIBVP}  locally in time for $0\leq s < \frac32$, $s \neq \frac12$ (the precise meaning of a solution is discussed in Section \ref{sec:notion}). Holmer extended this to $-\frac34 <s <\frac32$, $s \neq \frac12$ in \cite{holmer}. The technique used by Colliander-Kenig involves extending the initial data and recasting the problem as an IVP on $\R$ after introducing a Riemann-Liouville fractional integral forcing operator to satisfy the boundary condition. The proof then proceeds by a standard fixed point argument in the Bourgain spaces $X^{s,b}$ with a few important modifications. The half line theory requires that $b < \frac12$ to close the contraction argument, whereas $b>\frac12$ can be used on the full line. Indeed, the bilinear estimate \eqref{bil_kdv} fails for $b<\frac12$, with counterexamples arising due to poor control of the lower frequencies. To address this, Colliander-Kenig carry out the argument in a modified Bourgain space, $X^{s,b} \cap V^\gamma$, instead. We note that Colliander-Kenig and Holmer do not establish uniqueness, which is more difficult on the half line, especially at low regularity, as one must show that the solution does not depend on the choice of extension of the initial data.

In \cite{bona}, Bona-Sun-Zhang proposed an alternative approach via the Laplace transform which essentially separates the problem into a \textit{linear} IBVP on $\R^+$ and a nonlinear IVP on $\R$ after extending the initial data. With this technique an explicit solution to the linear IBVP is constructed, so the Riemann-Liouville forcing term of Colliander-Kenig is not needed. Many of the results and estimates from well-posedness arguments on $\R$ can then be applied to the nonlinear problem, with some important modifications. We describe this in more detail in the next section. Erdogan-Tzirakis have applied this technique to establish global well-posedness and smoothing properties for the cubic NLS on the half line \cite{tz_cubicnls}. Similar arguments have been carried out for the ``good" Boussinesq equation \cite{tz_good}, the Zakharov system \cite{tz_zakharov}, the derivative NLS \cite{tz_derivnls}, and the Klein-Gordon-Schrodinger sytem \cite{tz_klein}. We intend to follow this method to establish low regularity well-posedness for the Majda-Biello system \eqref{majdaIBVP}. They key step is the proof of the bilinear estimates in Section \ref{sec: bilinear}. 

Another approach to nonlinear IBVPs on the half line is the unified transform method of Fokas \cite{fokas}. In \cite{himonas}, Fokas, Himonas, and Mantzavinos applied this method to establish local well-posedness to the KdV IBVP \eqref{kdvIBVP} for $\frac34<s<1$. However, this method relies on inverse-scattering techniques and therefore requires complete integrability, which does not hold for our system \eqref{majdaIBVP}. The unified transform method also requires higher regularity, while we are interested in matching the low regularity from the local theory of the Majda-Biello IVP \eqref{majdaIVP} on the full line. 
 
We now outline the organization of this paper. In Section \ref{sec:notion} we discuss the $H^s(\R^+)$ spaces and properties related to extending the initial and boundary data. We also illustrate the Laplace transform method of \cite{bona}, define the explicit notion of a solution to an IBVP, and state the main theorem. In Section \ref{sec:apriori} we review the estimates on the linear terms.  In Section \ref{sec: bilinear} we establish the main bilinear estimates. We then prove Theorem \ref{thrm:main} in Section \ref{sec:proof}, beginning with the local theory in \ref{subsec:local} via the contraction fixed-point argument, followed by a discussion of uniqueness for $s>\frac32$ in \ref{subsec:uniqueness}.

\subsection{Notation}

The following notation will be used throughout the paper.\\
We define the Fourier transform on $\R$ as
$$\widehat{g}(\xi) = \mathcal{F}g(\xi) = \int_{\R} e^{-ix \cdot \xi} g(x) \, dx$$
and likewise the space-time Fourier transform:
$$\widehat{g}(\xi,\tau) = \mathcal{F}g(\xi,\tau) = \int_{\R^2} e^{-ix \cdot \xi} e^{-it\cdot \tau} g(x,t) \, dx dt.$$

\noindent We define the Sobolev space $H^s(\R)$ via the norm:
$$\|g\|_{H^s(\R)} = \|\la\xi\ra^s \widehat{g}(\xi) \|_{L^2(\R)}$$
where $\la\xi\ra := 1 + |\xi|$ (or equivalently $\sqrt{1+\xi^2}$).\\

\noindent We write $W^t u \; (x,t)$ and $W^t_\alpha v \;(x,t)$ for the linear Airy propagators:
\begin{equation*}
\begin{split}
W^t u \;(x,t) &= e^{-t\partial_{xxx}}u = \int e^{ix\xi} e^{it\xi^3}\widehat{u}(\xi,t) d\xi\\
W^t_\alpha v \; (x,t) &= e^{-\alpha t \del_{xxx}} v = \int e^{ix\xi} e^{i\alpha t\xi^3}\widehat{v}(\xi,t) d\xi.
\end{split}
\end{equation*}

\noindent For a space-time function we write $D_0$ for evaluation at the boundary $x=0$. By Fourier inversion we have
$$D_0(f(x,t)) = f(0,t) = \frac{1}{4\pi^2}\int_{\R^2} e^{it \tau} \wh{f}(\xi,\tau) d\xi d\tau$$

\noindent We will use the mixed norm notation for several a priori estimates
$$\|u\|_{L^p_xL^q_t} = \Big\| \|u\|_{L^q_t}\Big\|_{L^p_x},$$
and we make repeated use of the restricted norm spaces of Bourgain
\begin{align*}
\|u\|_{X^{s,b}} &= \|\la\xi\ra^s \la\tau-\xi^3\ra^b \widehat{u}(\xi,\tau)\|_{L^2_{\xi \tau}}\\
\|u\|_{X^{s,b}_\alpha} &= \|\la\xi\ra^s \la\tau-\alpha\xi^3\ra^b \widehat{u}(\xi,\tau)\|_{L^2_{\xi \tau}},
\end{align*}
as well as the low frequency modification space $V^\gamma$ introduced by Colliander-Kenig in \cite{colliander}
$$\|u\|_{V^\gamma} = \|\chi_{|\xi| \leq 1}\la\tau\ra^\gamma \widehat{u}(\xi,\tau)\|_{L^2_{\xi\tau}}, \quad \gamma > \frac12.$$

We write $\eta(t)$ for a $C^\infty_c (\R)$ function (smooth function with compact support) which is equal to $1$ on $[-1,1]$. Finally, we make use of the conventional notation $a \lsim b$, meaning  $a \leq C \, b$ for some absolute constant $C$. We define $a \gsim b$ similarly, and we write $a\sim b$ when $a \lsim b \lsim a$. 

\section{Notion of a solution}\label{sec:notion}

In this section we make precise the notion of a solution to the IBVP \eqref{majdaIBVP}, and what it means for a solutions to be locally well-posed.
We aim to find a solution to \eqref{majdaIBVP} with the additional compatibility condition $f(0) = u_0(0)$ and $g(0) = v_0(0)$ if $s>\frac12$. We will reformulate \eqref{majdaIVP} as an integral equation. 

We start by choosing extensions for the initial data such that $\|\tu_0\|_{H^s(\R)} \lsim \|u_0\|_{H^s(\R^+)}$ and $\|\tv_0\|_{H^s(\R)} \lsim \|v_0\|_{H^s(\R^+)}$. Following \cite{bona}, we split \eqref{majdaIVP} into two simpler problems. First, we have the nonlinear IVP on the full line
\begin{equation} \label{nonlinIVP}
\begin{cases}
u_t+u_{xxx} + vv_x = 0\\
v_t + \alpha v_{xxx} + (uv)_x = 0 \qquad x \in \R \\
(u,v)|_{t=0} = (\tu_0, \tv_0) \in H^s(\R).
\end{cases}
\end{equation}
By the Duhamel principle, up to a local existence time $T>0$, smooth solutions to \eqref{nonlinIVP} satisfy
\begin{align*}
u &= \eta(t)e^{-t \del_{xxx}} \tu_0 + \eta(t)\int_0^t e^{-(t-t') \del_{xxx}} F(u,v) \, dt' = \eta(t)W^t \tu_0 + \eta(t)\int_0^t W^{t-t'} F dt'\\
v &= \eta(t)e^{-\alpha t \del_{xxx}} \tv_0 + \eta(t)\int_0^t e^{-\alpha(t-t')\del_{xxx}} G(u,v) \, dt' = \eta(t)W^t_\alpha \tv_0 + \eta(t)\int_0^t W_\alpha^{t-t'} G dt' 
\end{align*}
where $F(u,v) = \eta(t/T) vv_x$, $G(u,v) = \eta(t/T) (uv)_x$, and $\eta(t) \in C^\infty_c(\R)$ and is identically $1$ on $[-1,1]$ . 
The smooth cutoff functions $\eta$ are important for closing the contraction argument by keeping each term on the right in an appropriate Banach space (given later in Definition \ref{def:LWP}).

We also have the \textit{linear} IBVP with zero initial data:
\begin{equation} \label{linIBVP}
\begin{cases}
u_t + u_{xxx} = 0 \\
v_t + \alpha v_{xxx} = 0 \qquad x, t \in \R^+ \\
(u,v)|_{t=0} = (0,0) \in H^s(\R^+), \qquad (u,v)|_{x=0} = (f - p, g-q) \in H^{\frac{s+1}{3}}(\R^+)
\end{cases}
\end{equation}
where $p$ and $q$ account for the boundary value of the nonlinear solution to \eqref{nonlinIVP}, ensuring the compatibility condition holds when $s>\frac12$
\begin{equation*}
\begin{split}
p &= \eta(t) D_0(W^t \tu_0) + \eta(t) D_0 \left(\int_0^t W^{t-t'} F dt'\right)\\
q &= \eta(t) D_0(W^t_\alpha \tv_0) + \eta(t) D_0 \left(\int_0^t W_\alpha^{t-t'} G dt'\right).
\end{split}
\end{equation*}
Note that \eqref{linIBVP} is decoupled because the nonlinearity is not included. By a formal application of the Laplace transform, described in section \ref{sec:boundary} below, solutions of \eqref{linIBVP} can be found explicitly. If we define
\begin{equation} \label{boundary}
W_1h (x,t) := \frac{3}{2\pi} \int \limits_0^\infty e^{\beta [-\frac{\sqrt{3}}{2} - \frac{1}{2}i]x} e^{i\beta^3 t} \rho(\beta x) \beta^2 \widehat{h}(\beta^3) d\beta \end{equation}

then \eqref{linIBVP} has solution
\begin{align*}
u &= 2\Re[W_1(f-p) (x,t)] \\
v &= 2 \Re[W_1(g-q) (\sqrt[3]{\alpha}\,x,t)].
\end{align*}

As discussed in section \ref{sec:boundary}, we take $\widehat{h}$ to mean $\mathcal{F}_t [\chi_{[0,\infty)}h ]$. Since we are not free to choose any extension $\widetilde{h}$ of the boundary data as we did with the initial data, we will need an estimate on the size of $\|\chi_{(0,\infty)}h\|_{H^{\frac{s+1}{3}}(\R)}$. For this, we appeal to lemma \ref{extensions}, (see \cite{tz_cubicnls}, or \cite{colliander} for a full discussion of these half line Sobolev spaces). For $0\leq \frac{s+1}{3} <\frac12$, the extension by zero is controlled by the original data. In fact, since we can define the Sobolev spaces with negative index by duality, this also applies for $-\frac12<\frac{s+1}{3}<\frac12$. For $\frac12<\frac{s+1}{3}<\frac32$, the fact that $h(0) = 0$ is a direct consequence of the compatibility condition we imposed in \eqref{majdaIVP}. We remark that the case $\frac{s+1}{3} = \frac12$ (i.e. $s = \frac12$) is excluded in Lemma \ref{extensions}. This is due to difficulty in formulating the compatibility with the trace operator.

\begin{lemma}\label{extensions} Let $h \in H^s(\R^+)$ for some $-\frac{1}{2} < s < \frac{3}{2}.$\\

\noindent i) If $-\frac{1}{2}<s<\frac{1}{2}$, then $\norm{\chi_{(0,\infty)}h}_{H^s(\R)} \lsim \norm{h}_{H^s(\R^+)}$ \\
ii) If $\frac{1}{2}<s<\frac{3}{2}$ and $h(0)=0$, then $\norm{\chi_{(0,\infty)}h}_{H^s(\R)} \lsim \norm{h}_{H^s(\R^+)}.$\\
\end{lemma}

Combining the solutions to \eqref{nonlinIVP} and \eqref{linIBVP} above, we arrive at the integral formulation of our IBVP \eqref{majdaIBVP}

\begin{equation} \label{weaksol}
\begin{split}
u(x,t) &= \eta(t)W^t \tu_0 + \eta(t)\int_0^t W^{t-t'} F(x,t') \,dt' + 2\Re W_1(f-p) (x,t) \\
v(x,t) &= \eta(t)W^t_\alpha \tv_0 + \eta(t)\int_0^t W_\alpha^{t-t'} G(x,t')\, dt' + 2\Re W_1(g-q) (\sqrt[3]{\alpha}\,x,t) 
\end{split}
\end{equation}

where:\\
\\
$
\begin{array}{lll}
F(u,v) = \eta(t/T) vv_x & \quad&  G(u,v) = \eta(t/T) (uv)_x\\
p = \eta(t) D_0(W^t \tu_0) + \eta(t) D_0 \left(\int_0^t W^{t-t'} F dt'\right) &\quad &
q = \eta(t) D_0(W^t_\alpha \tv_0) + \eta(t) D_0 \left(\int_0^t W_\alpha^{t-t'} G dt' \right).
\end{array}
$\\

We are finally ready to define a solution to \eqref{majdaIBVP} and state or main theorem. 

\begin{definition}\label{def:LWP} We say \eqref{majdaIBVP} is \textit{locally well-posed} in $H^s(\R^+)$ if, for any $u_0, v_0 \in H^s(\R^+)$ and any $f,g \in H^{\frac{s+1}{3}}(\R^+)$, with the additional compatibility conditions $u_0(0) = f(0)$ and $v_0(0) = g(0)$ for $s>\frac12$, the system \eqref{weaksol} has a unique solution $(u,v)$ with 
\begin{align*}
u &\in C^0_tH^s_x([0,T] \times) \cap C^0_xH^{\frac{s+1}{3}}_t(\R \times [0,T]) \cap X^{s,b}(\R \times [0,T]) \cap V^\gamma(\R \times [0,T]) \\
v &\in C^0_tH^s_x([0,T] \times) \cap C^0_xH^{\frac{s+1}{3}}_t(\R \times [0,T]) \cap X_\alpha^{s,b}(\R \times [0,T]) \cap V^\gamma(\R \times [0,T]) 
\end{align*}
for some $b<\frac12$ and $\gamma>\frac12$. Moreover, if $(u,v)$ and $(u', v')$ are two such solutions obtained with the same initial data $(u_0,v_0)$, then the two solutions are equal on $[0,\infty) \times [0,T]$. Further, if $u_{0,n} \to u_0$ and $v_{0,n} \to v_0$ in $H^s(\R^+)$ and $f_n \to f$ and $g_n \to g$ in $H^{\frac{s+1}{3}(\R^+)}$, then $u_n \to u$ and $v_n \to v$ in the spaces above. 
\end{definition}

\begin{theorem}\label{thrm:main}
Let $0<\alpha<1$ and $0<s<2$, $s \neq \frac12, \frac32$. Let $u_0, v_0$ be in $H^s(\R^+)$ and $f,g$ be in $H^{\frac{s+1}{3}}(\R^+)$. If $s>\frac12$, also let $u_0(0) = f(0)$ and $v_0(0) = g(0)$ if $s>\frac12$. Then there exists $T>0$, $b<\frac12$, $\gamma>\frac12$, and $u$, $v$ satisfying
\begin{align*}
u &\in C^0_tH^s_x([0,T] \times) \cap C^0_xH^{\frac{s+1}{3}}_t(\R \times [0,T]) \cap X^{s,b}(\R \times [0,T]) \cap V^\gamma(\R \times [0,T]) \\
v &\in C^0_tH^s_x([0,T] \times) \cap C^0_xH^{\frac{s+1}{3}}_t(\R \times [0,T]) \cap X_\alpha^{s,b}(\R \times [0,T]) \cap V^\gamma(\R \times [0,T]) 
\end{align*}
such that $(u,v)$ is a distributional solution to the initial-boundary value problem \eqref{majdaIBVP}. That is, $(u,v$) solves \eqref{weaksol} in the distributional sense. Moreover, if $\frac32<s<2$, then $(u,v)$ is a locally well-posed solution to \eqref{majdaIBVP}.
 \end{theorem}

The lower regularity bound $s>0$ is almost-sharp, as Oh showed in \cite{oh} that the Majda-Biello problem on $\R$ is ill-posed for $s<0$ because the data-to-solution map is not $C^2$. The maximum regularity imposed by the Laplace transform method would be $s<\frac72$ so that $\frac{s+1}{3}<\frac32$. 

As in \cite{kenig} and \cite{holmer}, we do not establish uniqueness for low regularity solutions. If we had smoothing in our bilinear estimates in Section \ref{sec: bilinear}, we could prove uniqueness by following the approach from \cite{tz_cubicnls} or \cite{tz_good}. However, the local theory for the KdV equation on $\R$ does not appear to exhibit smoothing in the bilinear estimate (see \cite[Proposition 1]{colliander_no_smooth}, for example), so we do not expect smoothing for the Majda-Biello system on the half line either. Perhaps a one-sided smoothing estimate could be used to establish uniqueness. We plan to investigate this in future work. 

\section{Boundary Term} \label{sec:boundary}
Here we give a (formal) derivation of the boundary term solution of \eqref{linIBVP} using the Laplace transform. As the equations are decoupled, we solve only the first equation here. Our goal is an explicit solution of 
\begin{equation} \label{Linear_IBVP}
\begin{split}
&u_t + u_{xxx} = 0 \qquad x \in \R^+, \, t \in \R^+ \\
&u(x,0) \equiv 0,  \quad u(0,t) = h(t)
\end{split}
\end{equation}
for $h \in H^{\frac{s+1}{3}}(\R^+)$.  We'll denote this solution by $W_0 h (x,t)$.\\
\\
Taking the Laplace transform in $t$  gives us the ordinary IVP (in $x$)
\begin{equation}
\begin{split}
&\tilde{u}_{xxx} + \lambda \tilde{u} = 0 \\
&\tilde{u}(0,\lambda) = \tilde{h}(\lambda)
\end{split}
\end{equation}
which is easily solved to give 
$$\tilde{u}(x,\lambda) = e^{r(\lambda)x}\tilde{h}(\lambda) \qquad \text{where } r^3(\lambda) + \lambda = 0.$$
\noindent Now inverting with the Mellin transform allows us to write
\begin{align*}
W_0 h (x,t) &= \chi_{[0,\infty)}(t) \, u(x,t) = \lim_{\gamma \to 0} \frac{1}{2 \pi i} \int \limits_{\gamma-i\infty}^{\gamma + i\infty} e^{\lambda t} \tilde{u}(x,\lambda) d\lambda \\
 &= \frac{1}{2\pi} \int \limits_0^\infty e^{i\beta^3 t} \tilde{u}(x,i\beta^3)3\beta^2 d\beta + \frac{1}{2\pi} \int \limits_0^\infty e^{-i\beta^3 t} \tilde{u}(x,-i\beta^3)3\beta^2 d\beta \\
 &:= W_1h + W_2h
\end{align*}

\noindent For $W_1h$, we have $r^3 + i\beta^3 = 0$. The only root with $\text{Re}(r)<0$ is $\beta\left[-\frac{\sqrt{3}}{2} - \frac{1}{2}i\right]$, so 
\begin{align*}
W_1 h(x,t) &= \frac{3}{2\pi} \int \limits_0^\infty e^{\beta [-\frac{\sqrt{3}}{2} - \frac{1}{2}i]x} e^{i\beta^3 t} \beta^2 \tilde{h}(i\beta^3) d\beta \\
 &= \frac{3}{2\pi} \int \limits_0^\infty e^{\beta [-\frac{\sqrt{3}}{2} - \frac{1}{2}i]x} e^{i\beta^3 t} \beta^2 \widehat{h}(\beta^3) d\beta
\end{align*}
where we've abused notation slightly in writing $\widehat{h}$ to represent $\mathcal{F}_t [\chi_{[0,\infty)}h ]$.\\

To extend $W_1h$ to all $x$, we introduce a smooth function $\rho$ supported in $(-2,\infty)$ with $\rho \equiv 1$ on $[0,\infty)$
$$W_1h (x,t) = \frac{3}{2\pi} \int \limits_0^\infty e^{\beta [-\frac{\sqrt{3}}{2} - \frac{1}{2}i]x} e^{i\beta^3 t} \rho(\beta x) \beta^2 \widehat{h}(\beta^3) d\beta.$$

A similar calculation shows
$$W_2h (x,t) = \overline{W_1h} (x,t)  = \frac{3}{2\pi} \int \limits_0^\infty e^{\beta [-\frac{\sqrt{3}}{2}+\frac{1}{2}i]x} e^{-i\beta^3t} \rho(\beta x) \beta^2 \widehat{h}(-\beta^3) d\beta.$$

\section{A priori estimates} \label{sec:apriori}

In this section we verify that linear terms in \eqref{weaksol} remain in the Banach space from definition \ref{def:LWP}. Several of the estimates in the section are standard properties of $X^{s,b}$ spaces, and some are well known linear estimates from the local theory for the KdV equation on $\R$. 

We begin with the linear solution $W^t u_0$. Recall that $\eta \in C^\infty_c$ is a smooth cutoff equal to $1$ on $[-1,1]$. Because $W^t = e^{-t\del_{xxx}}$ is unitary on $H^s$, we know
\begin{equation} \label{W_unitary}
\|\eta(t) W^t u_0\|_{L^\infty_t H^s_x} \lsim \|u_0\|_{H^s_x(\R)}.
\end{equation}
Then $\eta(t) W^t u_0 \in C^0_tH^s_x(\R)$ by the dominated convergence theorem.

Next, we have the well known Kato smoothing estimate \cite[Lemma 4.1]{kenig}.
\begin{equation} \label{Kato_smooth}
\|\partial_x W^t u_0 \|_{L^\infty_x L^2_t} \lsim \|u_0\|_{L^2(\R)}.
\end{equation}

To take full advantage of the interplay between space and time derivative, we have the following Kato type estimate. Note that this explains our choice to take the boundary data to be in $H^{\frac{s+1}{3}}(\R^+)$.

\begin{lemma} \cite[Lemma 4.1]{kenig}\label{lem_Kato}
For ($s\geq-1$)
$$\| \eta(t) W^t u_0 \|_{L_{x}^{\infty} H_{t}^{\frac{s+1}{3}}} \lsim \| u_0 \|_{H^s(\R)}.$$
\end{lemma}

Straightforward estimates show that the linear solution lies in $X^{s,b} \cap V^\gamma$. Recall that $V^\gamma$ is the low-frequency adjustment introduced in \cite{kenig}.

\begin{lemma} \cite[Lemma 5.2]{kenig} \label{lin_V} For any $\gamma \in \R$,
$$\|\eta(t) W^t u_0\|_{V^\gamma} \lsim \|u_0\|_{L^2 (\R)}.$$
\end{lemma}

\begin{lemma}\cite[Lemma 2.8]{tao_nonlinear}\label{xHs_embed}
$$\| \eta(t) W^t u_0 \|_{X^{s,b}} \lsim \| u_0 \|_{H_x^s(\R)} \qquad \text{for any }s,b \in \R.$$
\end{lemma}

Although we must take $b<\frac12$ in Theorem \ref{thrm:main}, we will also make use of the following standard result \cite[Corollary 2.10]{tao_nonlinear}.
\begin{equation} \label{Xsb_embed}
\text{For any }s \in \R \text{ and } b>\frac12, \text{ we have }X^{s,b} \subset C^0_tH^s_x.
\end{equation}

We now proceed with the estimates for the nonlinear Duhamel term $\eta(t)\int_0^t W^{t-t'} F(x,t') dt' $, and the boundary term $W_1(h)$ defined in \eqref{boundary}. These estimates follow the general approach of Bona-Sun-Zhang in \cite{bona} and can be adapted to many initial-boundary value problems. To complete the argument, we will also need the bilinear estimates specific to the Majda-Biello IBVP \eqref{majdaIBVP}, which we establish in section \ref{sec: bilinear}.

\begin{lemma}\cite[Lemma 3.3]{Kenig93}\label{Duham} For $s \in \R$, $0 \leq b_1 < \frac12$, and $0 \leq b_2 \leq 1-b_1$\\
$$\|\eta(t) \int_0^t W^{t-t'} F dt' \|_{X^{s,b_2}} \lsim \|F\|_{X^{s,-b_1}}.$$
\end{lemma}
\noindent Remark: We will take $b_2 > \frac12$ in lemma \ref{Duham}, so this will put the Duhamel term in $C^0_tH^s_x$ as well by the embedding $X^{s,b} \subset C^0_tH^s_x$ for $b>\frac12$. 

\begin{lemma} \cite[Lemma 2.11]{tao_nonlinear} \label{extractT} For $T<1$, and $-\frac12<b_1<b_2<\frac12$, we have
$$\|\eta(t/T)F\|_{X^{s,b_1}} \lsim T^{b_2-b_1}\|F\|_{X^{s,b_2}}.$$
\end{lemma}

We need the following proposition to ensure that the Duhamel term stays in $C^0_xH^{\frac{s+1}{3}}_t$. Note that Proposition \ref{Duham_time_trace} is the reason we must work with $b<\frac12$ on the half-line. 

\begin{prop}[Duhamel Time Trace] \label{Duham_time_trace}
For $b<\frac{1}{2}$,
$$\norm{\eta(t) \int_0^t W^{t-t'} F dt'}_{C^0_x H^{\frac{s+1}{3}}_t} \lsim 
\begin{cases}
\norm{F}_{X^{s,-b}} & 0 \leq s \leq \frac{1}{2}\\
\displaystyle \norm{F}_{X^{s,-b}} + \|F\|_{X^{\frac12+, \frac{s-2}{3}}} & s>\frac12.
\end{cases}$$
\end{prop}

\begin{proof}
Suppose first that $0 \leq s \leq \frac12$. We begin by rewriting $\int_0^t W^{t-t'} F dt'$
\begin{align*}
\int_0^t W^{t-t'} F dt' (x,t) &= \int_0^t \int e^{ix\xi}e^{i(t-t')\xi^3}\widehat{F}(\xi,t') d\xi dt' 
=\int_{\xi} \int_0^t e^{ix\xi}e^{i(t-t')\xi^3} \left( \int_{\tau} e^{it'\tau} \widehat{F}(\xi,\tau) d\tau \right) dt' d\xi \\
&= \iint e^{ix\xi}e^{it\xi^3} \widehat{F}(\xi,\tau) \left( \int_0^t e^{it'(\tau-\xi^3)} dt' \right) d\xi d\tau 
= \iint e^{ix\xi} \frac{e^{it\tau}-e^{it\xi^3}}{i(\tau-\xi^3)} \widehat{F}(\xi,\tau) d\xi d\tau.
\end{align*}

We now proceed to bound 
$$\norm{\eta \iint e^{ix\xi} \frac{e^{it\tau}-e^{it\xi^3}}{i(\tau-\xi^3)} \widehat{F}(\xi,\tau) d\xi d\tau }_{H^{\frac{s+1}{3}}_t}.$$
We first consider the case $0 \leq s \leq \frac{1}{2}$, where we'll make repeated use of the inequality 
\begin{equation} \label{Sob_product}
\norm{uv}_{H^s} \lsim \norm{u}_{H^1} \norm{v}_{H^s}
\end{equation}
which follows trivially from Lemma \ref{A.1}. We treat the cases where $|\tau - \xi^3| \leq 1$ and $|\tau -\xi^3|>1$ separately.\\ 
\\
\emph{Case 1)} $|\tau - \xi^3| \leq 1$ \\

Taylor expanding the exponentials gives 
$$ \frac{e^{it\tau}-e^{it\xi^3}}{i(\tau-\xi^3)} = i e^{it\tau} \frac{1}{\tau-\xi^3} (e^{-it(\tau-\xi^3)}-1) = i e^{it\tau} \sum_{k=1}^{\infty} \frac{(-it)^k(\tau-\xi^3)^{k-1}}{k!}$$

so using (\ref{Sob_product}) and the Liebniz rule (recall $\eta \in C_c^\infty(\R)$), we have
\begin{align*}
&\norm{\eta \iint_{|\tau-\xi^3|\leq1} e^{ix\xi} \sum_{k=1}^{\infty} ie^{it\tau} \frac{(-it)^k}{k!}(\tau-\xi^3)^{k-1} \widehat{F}(\xi,\tau) d\xi d\tau }_{H^{\frac{s+1}{3}}_t}\\
&\quad\lsim \sum_{k=1}^{\infty} \frac{\norm{(-it)^k \eta}_{H^1}}{k!} \norm{\iint_{|\tau-\xi^3|\leq1} e^{ix\xi}e^{it\tau} (\tau-\xi^3)^{k-1} \widehat{F}(\xi,\tau) d\xi d\tau }_{H^{\frac{s+1}{3}}_t} \\
&\quad\lsim \sum_{k=1}^{\infty} \frac{1}{(k-1)!} \norm{\jbrac{\tau}^{\frac{s+1}{3}} \int_{|\tau-\xi^3|\leq1} e^{ix\xi} (\tau-\xi^3)^{k-1} \widehat{F}(\xi,\tau) d\xi}_{L^2_\tau}
\end{align*}
where we've used $\|t^k \eta \|_{H^1} \lsim \|kt^{k-1}\eta\|_{L^2} + \|t^k \eta'(t)\|_{L^2} \lsim k$.

Since $\displaystyle \sum_{k=1}^{\infty} \frac{1}{(k-1)!}$ converges, it is enough to bound

\begin{align*}
\norm{\jbrac{\tau}^{\frac{s+1}{3}} \int |\widehat{F}(\xi,\tau)| d\xi }_{L^2_\tau} &\lsim \left[ \int \jbrac{\tau}^{\frac{2(s+1)}{3}} \left(\int_{|\tau-\xi^3|\leq1} \jbrac{\xi}^{-2s} d\xi \right) \left( \int_{|\tau-\xi^3|\leq1} \jbrac{\xi}^{2s}|\widehat{F}|^2 d\xi \right) d\tau \right]^{1/2} \\
&\quad\lsim \sup_{\tau} \left[\jbrac{\tau}^{\frac{2(s+1)}{3}} \int_{|\tau-\xi^3|\leq1} \jbrac{\xi}^{-2s} d\xi \right]^{\frac12} \norm{F}_{X^{s,-b}}.
\end{align*}

\noindent The supremum is clearly bounded for $|\tau|\leq 2$. For $|\tau|>2$  we can change variables $(\rho = \xi^3)$ to bound it by
$$\jbrac{\tau}^{\frac{2(s+1)}{3}} \int \limits_{|\tau|-1}^{|\tau|+1} \jbrac{\rho}^{\frac{-2s}{3}} \frac{1}{3\rho^{2/3}} d\rho
\lsim \jbrac{\tau}^{\frac{2(s+1)}{3}}  \int \limits_{|\tau|-1}^{|\tau|+1} \rho^{\frac{-2(s+1)}{3}} d\rho \leq 2$$
since $1<|\rho| \sim \jbrac{\tau}$ here.\\
\\
\emph{Case 2)} $|\tau - \xi^3| > 1$\\

In this case we separate the $H_t$ norm into two terms
\begin{align*}
&\norm{\eta \iint_{|\tau-\xi^3|>1} e^{ix\xi} \frac{e^{it\tau}-e^{it\xi^3}}{i(\tau-\xi^3)} \widehat{F}(\xi,\tau) d\xi d\tau }_{H^{\frac{s+1}{3}}_t}\\
&\quad \leq \norm{\eta \iint_{|\tau-\xi^3|>1} e^{ix\xi} \frac{e^{it\tau}}{\tau-\xi^3} \widehat{F}(\xi,\tau) d\xi d\tau }_{H^{\frac{s+1}{3}}_t}+
\norm{\eta \iint_{|\tau-\xi^3|>1} e^{ix\xi} \frac{e^{it\xi^3}}{\tau-\xi^3} \widehat{F}(\xi,\tau) d\xi d\tau }_{H^{\frac{s+1}{3}}_t}\\
&\quad  := \norm{I} + \norm{II}
\end{align*}

\noindent For $\norm{I}$ we use (\ref{Sob_product}) immediately:
\begin{align*}
\norm{I} &\lsim \norm{\eta}_{H^1_t} \norm{\jbrac{\tau}^{\frac{s+1}{3}} \int_{|\tau-\xi^3|>1} \frac{1}{\jbrac{\tau-\xi^3}} |\widehat{F}(\xi,\tau)| d\xi}_{L^2_\tau}\\
 & \lsim \left[\int \jbrac{\tau}^{\frac{2(s+1)}{3}} \left( \int \frac{d\xi}{\jbrac{\xi}^{2s}\jbrac{\tau-\xi^3}^{2-2b}} \right) \left(\int \jbrac{\xi}^{2s} \jbrac{\tau-\xi^3}^{-2b} |\widehat{F}(\xi,\tau)|^2 d\xi  \right)d\tau \right]^{1/2}\\ 
 &\leq \sup_{\tau} \left[ \jbrac{\tau}^{\frac{2(s+1)}{3}} \int \frac{d\xi}{\jbrac{\xi}^{2s}\jbrac{\tau-\xi^3}^{2-2b}} \right]^{1/2} \norm{F}_{X^{s,-b}}
 \lsim \norm{F}_{X^{s,-b}}.
\end{align*}
The above supremum is finite because
\begin{align*}
&\jbrac{\tau}^{\frac{2(s+1)}{3}} \int \limits_{|\xi|\leq1} \frac{d\xi}{\jbrac{\xi}^{2s}\jbrac{\tau-\xi^3}^{2-2b}} + 
\jbrac{\tau}^{\frac{2(s+1)}{3}} \int \limits_{|\xi|>1} \frac{d\xi}{\jbrac{\xi}^{2s}\jbrac{\tau-\xi^3}^{2-2b}}\\
& \quad \lsim \jbrac{\tau}^{\frac{2(s+1)}{3}} \frac{1}{\jbrac{\tau}^{2-2b}} \int \limits_{|\xi|\leq1} \frac{d\xi}{\jbrac{\xi}^{2s}}
+ \jbrac{\tau}^{\frac{2(s+1)}{3}} \int \limits_{|\rho|>1} \frac{d\rho}{\jbrac{\rho}^{\frac{2(s+1)}{3}} \jbrac{\tau-\rho}^{2-2b}}\\
& \quad \lsim \jbrac{\tau}^{\frac{2(s+1)}{3}} \frac{1}{\jbrac{\tau}^{2(1-b)}} +
\jbrac{\tau}^{\frac{2(s+1)}{3}} \jbrac{\tau}^{-\frac{2(s+1)}{3}} \lsim 1
\end{align*}
where we've used the calculus lemma \ref{A.2} in the $|\rho|>1$ term and the fact that $0 \leq \frac{s+1}{3} \leq \frac{1}{2}$, $b<\frac{1}{2}$.\\

For $\norm{II}$, we separate into $|\xi|\leq 1$ and $|\xi|>1$. For $|\xi|\leq1$, we use Minkowski's inequality followed by the Cauchy-Schwarz inequality
\begin{align*}
\norm{II_{|\xi|\leq1}} &\leq \int \limits_{|\tau-\xi^3|>1} \int \limits_{|\xi|\leq1} \norm{\eta e^{it\xi^3}}_{H^{\frac{s+1}{3}}_t} \frac{|\widehat{F}(\xi,\tau)|}{|\tau-\xi^3|} d\xi d\tau\\
&\lsim \iint \limits_{|\xi|\leq1} \frac{1}{\jbrac{\tau-\xi^3}}|\widehat{F}(\xi,\tau)| d\xi d\tau \\
&\lsim  \left[ \iint_{|\xi|\leq1} \frac{1}{\jbrac{\tau}^{2-2b}} d\xi d\tau \right]^{1/2} \norm{F}_{X^{s,-b}} \lsim \norm{F}_{X^{s,-b}}
\end{align*}
since $2-2b>1$ for $b<\frac12$. 
\\

For $|\xi|>1$, we start with \eqref{Sob_product} and the change of variables $\rho = \xi^3$:
\begin{align*}
\norm{II_{|\xi|>1}} &\lsim \norm{\eta}_{H^1_t} \norm{\int_{|\tau-\rho|>1} \int_{|\rho|>1} \frac{e^{ix\sqrt[3]{\rho}}e^{it\rho}}{i(\tau-\rho)}\widehat{F}(\sqrt[3]{\rho},\tau) \frac{1}{3\rho^{2/3}} d\rho d\tau }_{H^{\frac{s+1}{3}}_t}\\
& \lsim \norm{\jbrac{\rho}^{\frac{s+1}{3}} \mathcal{F}_t \left[ \mathcal{F}_\rho^{-1}\left( \int_{|\tau-\rho|>1,\, |\rho|>1}\frac{e^{ix\sqrt[3]{\rho}}}{i(\tau-\rho)} \widehat{F}(\sqrt[3]{\rho},\tau) \frac{1}{3\rho^{2/3}} d\tau \right)(t) \right](\rho)}_{L^2_\rho}\\
& \lsim \norm{\jbrac{\rho}^{\frac{s+1}{3}} \int \frac{1}{\jbrac{\tau-\rho}}|\widehat{F}(\sqrt[3]{\rho},\tau)|\frac{1}{\rho^{2/3}} d\tau }_{L^2_{|\rho|>1}}\\
& \lsim \left[\int \jbrac{\rho}^{\frac{2(s+1)}{3}} \frac{1}{\jbrac{\rho}^{2/3}} \left(\int \frac{1}{\jbrac{\tau-\rho}^{2-2b}} d\tau \right) \left(\int \frac{|\widehat{F}(\sqrt[3]{\rho},\tau)|^2}{\jbrac{\tau-\rho}^{2b}} d\tau \right) \frac{1}{|\rho|^{2/3}} d\rho \right]^{1/2}\\
& \lsim \left[\iint \jbrac{\rho}^{\frac{2s}{3}} \jbrac{\tau-\rho}^{-2b} |\widehat{F}(\sqrt[3]{\rho},\tau)|^2 \frac{1}{|\rho|^{2/3}} d\tau d\rho \right]^{1/2} \lsim \norm{F}_{X^{s,-b}}
\end{align*}

Next suppose $s>\frac12$ and proceed similarly. Rather than \eqref{Sob_product}, in this case we use the algebra property of Sobolev spaces
\begin{equation} \label{Sob_algebra}
\|u v \|_{H^s} \leq \|u\|_{H^s} \|v\|_{H^s}.
\end{equation}
Since $\eta$ is smooth, using \eqref{Sob_algebra} instead of \eqref{Sob_product} doesn't significantly alter the proof. The only remaining difference with the $0\leq s \leq \frac12$ case is that we needed $s\leq\frac12$ to bound $\|I\| \lsim \|F\|_{X^{s,-b}}$. For $s>\frac12$ we use the fact that $\la\tau\ra^{\frac{s+1}{3}} \lsim \la\tau-\xi^3\ra^{\frac{s+1}{3}} + |\xi|^{s+1}$ to write

\begin{align*}
\|I\| &= \left\| \eta \iint_{|\tau-\xi^3|>1} e^{ix\xi} \frac{e^{it\tau}}{\tau-\xi^3} \widehat{F}(\xi,\tau) \, d\xi d\tau \right\|_{H^{\frac{s+1}{3}}_t} \\
&\lsim \norm{\eta}_{H^1_t} \norm{\jbrac{\tau}^{\frac{s+1}{3}} \int \frac{1}{\jbrac{\tau-\xi^3}} |\widehat{F}(\xi,\tau)| d\xi}_{L^2_\tau} \\
&\lsim \left\| \int \la\tau-\xi^3\ra^{\frac{s-2}{3}} |\wh{F}(\xi,\tau)| d\xi\right\|_{L^2_\tau} + \left\| \int \frac{|\xi|^{s+1}}{\la\tau-\xi^3\ra} |\wh{F}(\xi,\tau)|d\xi \right\|_{L^2_\tau}.
\end{align*}

We apply Cauchy-Schwarz to each term. For the first term, we have
\begin{align*}
\left\| \int \la\tau-\xi^3\ra^{\frac{s-2}{3}} |\wh{F}(\xi,\tau)| d\xi\right\|_{L^2_\tau} &\lsim \left[ \int \left( \int \frac{1}{\la\xi\ra^{1+}} d\xi\right) \left(\int \la\xi\ra^{1+}\la\tau-\xi^3\ra^{\frac{2(s-2)}{3}} |\wh{F}(\xi,\tau)|^2 d\xi\right) d\tau \right]^{\frac12} \\
&\lsim \|F\|_{X^{\frac12+, \frac{s-2}{3}}}.
\end{align*}
And for the second term, since $b<\frac12$
\begin{align*}
\left\| \int \frac{|\xi|^{s+1}}{\la\tau-\xi^3\ra} |\wh{F}(\xi,\tau)|d\xi \right\|_{L^2_\tau} & \lsim \left[ \int \left( \int \frac{|\xi|^2}{\la\tau-\xi^3\ra^{2-2b}} d\xi\right) \left(\int \frac{\la\xi\ra^{2s}}{\la\tau-\xi^3\ra^{2b}} |\wh{F}(\xi,\tau)|^2 d\xi\right) d\tau \right]^{\frac12} \\
&\lsim \sup_\tau \left[ \int \frac{|\xi|^2}{\la\tau-\xi^3\ra^{2-2b}} d\xi \right]^\frac12 \, \|F\|_{X^{s,-b}} \\
&\lsim \sup_\tau \left[ \int \frac{1}{\la\tau-\rho\ra^{2-2b}} d\rho \right]^\frac12 \, \|F\|_{X^{s,-b}} \lsim \|F\|_{X^{s,-b}}.
\end{align*}

\end{proof}

We turn now to the boundary term $W_1(f-p)$ defined in \eqref{boundary}. 

\begin{prop} \label{Boundary_CtHx}
For $s \geq 0$, $h \in H^{\frac{s+1}{3}}(\R^+)$ satisfying $h(0) = 0$ if $s>\frac12$, 
\begin{equation*}
\norm{W_1(h)}_{C^0_t H^s_x} \lsim \norm{h}_{H^{\frac{s+1}{3}}(\R^+)}.
\end{equation*}
\end{prop}

\begin{proof}
 
We set $f(x) = e^{\left[-\frac{\sqrt{3}}{2}-\frac{i}{2}\right]x} \rho(x)$. Note that $f$ is a Schwartz function ($f \in \mathcal{S}$). Then we can write
$$W_1 h (x,t) = \frac{3}{2\pi} \int f(\beta x) \mathcal{F}_x [e^{t\del_{xxx}} \psi ] (\beta) d\beta$$
where 
$$\widehat{\psi} = \beta^2 \,  \widehat{h}(\beta^3) \chi_{[0,\infty)}(\beta).$$ Recall our notation of writing $\widehat{h}$ to mean $\mathcal{F}_t [\chi_{[0,\infty)}h ]$. We now show that $\norm{\psi}_{H^s} \lsim \norm{h}_{H^{\frac{s+1}{3}}_t(\R^+)}.$

\begin{align*}
\norm{\psi}_{H^s} &= \norm{\jbrac{\beta}^s \beta^2 \, \widehat{h}(\beta^3) }_{L^2_{\beta \geq 0}} = \sqrt{\int_0^\infty \jbrac{\beta}^{2s} \beta^2 |\widehat{h}(\beta^3)|^2 \beta^2 d\beta} \\
& \stackrel{(\rho = \beta^3)}{\lsim} \sqrt{\int_0^\infty \jbrac{\rho^{\frac{1}{3}}}^{2s} \rho^{\frac{2}{3}} |\widehat{h}(\rho)|^2 d\rho} = 
\norm{\jbrac{\rho^{\frac{1}{3}}}^s \rho^{\frac{1}{3}} \mathcal{F}_t[\chi_{[0,\infty)}h](\rho)}_{L^2_{\rho \geq 0}} \\
&\lsim \norm{\mathcal{F}_t[\chi_{[0,\infty)}h](\rho)}_{L^2_{0 \leq \rho \leq 1}} + \norm{\jbrac{\rho}^{\frac{s+1}{3}} \mathcal{F}_t[\chi_{[0,\infty)}h](\rho)}_{L^2_{\rho > 1}} \\
& \leq \norm{\chi_{[0,\infty)}h}_{L^2_t(\R)} + \norm{\chi_{[0,\infty)}h}_{H^{\frac{s+1}{3}}_t(\R)} 
 \lsim \norm{h}_{H^{\frac{s+1}{3}}_t (\R^+)}
\end{align*}
using Lemma \ref{extensions} in the last inequality, along with the fact that $\frac{s+1}{3} \geq 0$.\\

Since $e^{t\del_{xxx}}$ is continuous on $H^s$, it remains only to show that, for $f \in \mathcal{S}$
$$T g(x) := \int f(\beta x) \,\widehat{g}(\beta) d\beta$$
satisfies $\norm{Tg}_{H^s} \lsim \norm{g}_{H^s},$ where the implied constant may depend on $f$ (and its derivatives).\\

For $s = 0$, we have
\begin{align*}
|Tg(x)| &\stackrel{\beta x \to \beta}{\leq} \int |f(\beta \widehat{g}(x^{-1} \beta)| x^{-1} d\beta \qquad \text{for all }x \neq 0 \\
\norm{Tg}_{L^2_x} &\leq \int |f(\beta)| \norm{x^{-1} \widehat{g}( x^{-1}\beta ) }_{L^2_x} \, d\beta \leq \int |f(\beta)| \frac{1}{\sqrt{\beta}} d\beta \norm{g}_{L^2} \lsim \norm{g}_{L^2}
\end{align*}
using the fact that $f \in \mathcal{S}$ and
$$\norm{x^{-1} \widehat{g} ( x^{-1}\beta )}_{L^2_x} = \left[\int \frac{1}{x^2} |\widehat{g} (y)|^2 dy \right]^{1/2} = \frac{1}{\sqrt{\beta}} \norm{g}_{L^2}.$$

The case $s \in \mathbb{N}$ follows from the $s=0$ case because
$$\del_x (Tg) = \int f^{(s)}(\beta x) \beta^s \, \widehat{g}(\beta) d\beta.$$
Then by interpolation, we have $\norm{Tg}_{H^s} \lsim \norm{g}_{H^s}$ for all $s \geq 0$. Hence 
$$\norm{W_1h}_{H^s_x} \lsim \norm{e^{t\del_{xxx}} \psi}_{H^s_x} = \norm{\psi}_{H^s_x} \lsim \norm{h}_{H^{\frac{s+1}{3}}_t (\R^+)}.$$ 

Finally, continuity in $t$ follows from the dominated convergence theorem as usual. 

\end{proof} 

\begin{prop} \label{Boundary_CxHt}
For $s\geq 0$, $h \in H^{\frac{s+1}{3}}(\R^+)$ satisfying $h(0) = 0$ if $s>\frac12$,
$$\norm{\eta W_1 (h)}_{C^0_x H^{\frac{s+1}{3}}_t} \lsim \norm{h}_{H^{\frac{s+1}{3}}_t (\R^+)}.$$
\end{prop}

\begin{proof} As in the previous proposition, we write $W_1(h)$ as
$$ W_1 h = \frac{3}{2\pi} \int f(\beta x) \mathcal{F}_x \left[e^{t\del_{xxx}} \psi\right](\beta) d\beta $$
where 
$$f(x) = e^{\left[-\frac{\sqrt{3}}{2}-\frac{i}{2}\right]x} \rho(x) \text{, and } \widehat{\psi} = \beta^2 \widehat{h}(\beta^3) \chi_{[0,\infty)}(\beta). $$
Then
\begin{align*}
W_1 h &= \int \mathcal{F}_{\beta}[ f(\beta x)](y) e^{t\del_{yyy} \psi(y) dy}\\
 &= \int \frac{1}{x} \widehat{f}\left(\frac{y}{x}\right) e^{t\del_{yyy}} \psi(y) dy \\
 & \stackrel{y \to xy}{=} \int \widehat{f}(y) \left(e^{t\del_{yyy}} \psi  \right) (xy) dy.
\end{align*}

We use Minkowski's inequality in the $H^{\frac{s+1}{3}}$ norm and the fact that $f \in \mathcal{S}$ to bound
\begin{align*}
\norm{\eta W_1 h}_{H^{\frac{s+1}{3}}_t} & \leq \int |\widehat{f}(y)| \norm{\eta \left(e^{t\del_{yyy}}\psi\right)(xy)}_{H^{\frac{s+1}{3}}_t} (y) dy \\
 & \leq \sup_{y} \left[\norm{\eta \left(e^{t\del_{yyy}}\psi\right)(xy)}_{H^{\frac{s+1}{3}}_t} \right] \Vert \widehat{f} \, \Vert_{L^1} \\
 & \lsim \norm{\psi}_{H^s_y} \lsim \norm{h}_{H^{\frac{s+1}{3}}_t (\R^+)}.
\end{align*}

We used the Kato smoothing inequality (Lemma \ref{lem_Kato}) in the last line. The presence of the $x$ doesn't change the fact that the $H^{\frac{s+1}{3}}_t$ norm is uniformly bounded. The proof goes through exactly the same as in Lemma \ref{lem_Kato}, simply replacing $x$ with $y$ and $e^{ix\xi}$ with $e^{i(xy)\xi}$. Then continuity in x follows by the dominated convergence theorem as usual. 

\end{proof}

\begin{prop}\label{Boundary_Xsb}
For $s \geq 0$, $b\leq \frac12$, and $h \in H^{\frac{s+1}{3}}(\R^+)$ satisfying $h(0) = 0$ if $s>\frac12$,\\
$$\norm{\eta W_1(h)}_{X^{s,b}} \lsim \norm{h}_{H_t^{\frac{s+1}{3}}(\R^+)}.$$
\end{prop}

\begin{proof}
We take $f$ as above and recall the notation of writing $\widehat{h}$ for $\mathcal{F}_t [\chi_{[0,\infty)}h ]$. We may assume $b>0$ and, by interpolation, we may also assume $s \in 3 \mathbb{N}_0$.  In fact, since
\begin{align*}
\del_x^{(s)} \eta W_1 h &= \frac{3}{2\pi} \eta \int_0^\infty f^{(s)}(\beta x) e^{i \beta^3 t} \beta^{s+2} \, \widehat{h}(\beta^3) d\beta \\
 &= \frac{3}{2\pi} (-i)^{(s/3)}  \eta \int_0^\infty f^{(s)}(\beta x) e^{i \beta^3 t} \beta^2 \, \mathcal{F}_t [\chi_{(0,\infty)} \del_t^{(s/3)} h](\beta^3) d\beta,
\end{align*}
it is enough to prove the bound for $s = 0$. We have
$$\widehat{\eta W_1 h}(\xi,\tau) = \frac{3}{2\pi}\int \limits_0^\infty \widehat{\eta}(\tau-\beta^3) \widehat{f} (\xi/\beta) \beta \, \widehat{h}(\beta^3) d\beta,$$

and because $f\in \mathcal{S}$,
$$|\widehat{f}(\xi/\beta)| \lsim \frac{1}{\jbrac{\xi/\beta}^3} \lsim \frac{1}{1+|\xi/\beta|^3} = \frac{\beta^3}{\beta^3+|\xi|^3}.$$

Similarly, as $\eta \in C^\infty$ with compact support, we are certainly free to bound
$$|\wh{\eta}(\tau-\beta^3)| \lsim \la\tau-\beta^3\ra^{-3},$$
and thus
$$\norm{\eta W_1 h}_{X^{0,b}} \lsim \norm{\jbrac{\tau-\xi^3}^{b} \int_0^\infty \la\tau-\beta^3\ra^{-3} \frac{\beta^4}{\beta^3+|\xi|^3} |\widehat{h}(\beta^3)| d\beta }_{L^2_{\xi \tau}}.$$

We consider the regions where $\beta^3 + |\xi|^3 \leq1$ and $\beta^3 + |\xi|^3 \geq 1$ separately. For the former case, 

\begin{align*}
\norm{\jbrac{\tau}^{b} \int_0^1 \jbrac{\tau}^{-3} \frac{\beta^4}{\beta^3+|\xi|^3} |\widehat{h}(\beta^3)|d\beta}_{L^2_{|\xi|\leq1} L^2_\tau} &\lsim \|\la\tau\ra^{-(3-b)}\|_{L^2_\tau}\int_0^1 \norm{\frac{\beta^4}{\beta^3+|\xi|^3}}_{L^2_{|\xi|\leq1}} |\widehat{h}(\beta^3)| d\beta \\
&\lsim \int_0^1 \beta^{\frac{3}{2}} |\widehat{h}(\beta^3)| d\beta\\
& \stackrel{\rho = \beta^3}{=} \int_0^1 \rho^{-\frac{1}{6}} |\widehat{h}(\rho)| d\rho \\
& \lsim \norm{\chi_{(0,\infty)} h}_{L^2(\R)} \leq \|\chi_{(0,\infty)} h\|_{H^{\frac13}_t(\R)} \lsim \norm{h}_{H_t^{\frac{1}{3}}(\R^+)}
\end{align*}
where we've used Cauchy-Schwarz, the Plancharel identity, and Lemma \ref{extensions} in the last line.\\
In the latter case where $\beta^3 +  |\xi|^3>1$, we have $\jbrac{\tau-\xi^3} \lsim \jbrac{\tau-\beta^3}\jbrac{\beta^3 + |\xi|^3}$ and $\beta^3 + |\xi|^3 \sim \jbrac{\beta^3 + |\xi|^3}$, and thus
\begin{align*}
\norm{\int_0^\infty \jbrac{\tau-\beta^3}^{b-3} \frac{\beta^4}{(\beta^3+|\xi|^3)^{1-b}} |\widehat{h}(\beta^3)| d\beta}_{L^2_{\xi \tau}} &\lsim
\norm{\int_0^\infty \jbrac{\tau-\beta^3}^{b-3} \norm{\frac{\beta^4}{(\beta^3+|\xi|^3)^{1-b}}}_{L^2_\xi} |\widehat{h}(\beta^3)| d\beta}_{L^2_\tau} \\
& \lsim \norm{\int_0^\infty \jbrac{\tau-\beta^3}^{b-3} \beta^{\frac32+3b} |\widehat{h}(\beta^3)|d\beta}_{L^2_\tau} \\
& \stackrel{\rho = \beta^3}{\lsim} \norm{\int_0^\infty \jbrac{\tau-\rho}^{b-3} \rho^{b-\frac16} |\widehat{h}(\rho)| d\rho }_{L^2_\tau} \\
& \lsim \norm{\jbrac{\tau}^{-(3-b)}}_{L^1_\tau} \norm{\la\rho\ra^{\frac13} |\widehat{h}(\rho)|}_{L^2_\rho} \\
&\lsim \|\chi_{(0,\infty)} h\|_{H^{\frac13W}_t(\R)} \lsim \norm{h}_{H_t^{\frac{1}{3}}(\R^+)}
\end{align*}
where we've used Minkowski's inequality, Young's inequality (Lemma \ref{A.4}), and then Lemma \ref{extensions}. We also needed $b\leq\frac12$ in the third line  so that $b -\frac16 \leq \frac13$.

\end{proof}

\begin{cor} \label{Boundary_V}
For $0<s<\frac72$, $s\neq \frac12$, $s\neq\frac32$, $\frac12<\gamma\leq \frac12+\frac{s}{3}$, $h \in H^{\frac{s+1}{3}}(\R^+)$ satisfying $h(0) = 0$ if $s>\frac12$, 
$$\|\eta W_1(h)\|_{V^\gamma} \lsim \|h\|_{H^\frac{s+1}{3}_t(\R^+)}.$$
\end{cor}

\begin{proof}
We note that $\|\eta W_1(h)\|_{V^\gamma} \leq \|\eta W_1(h)\|_{X^{0,\gamma}}$. From here, all of the calculations from the proof of Proposition \ref{Boundary_Xsb} apply (with $b$ replaced by $\gamma$) up until the final line of the second case, where we no longer have $\gamma -\frac16 \leq \frac13$. That is, we now have
\begin{align*}
\|\eta W_1(h)\|_{V^\gamma} \leq \|\eta W_1(h)\|_{X^{0,\gamma}} \lsim \|h\|_{H^{\frac13}_t(\R^+)} + \left\|\int_0^\infty\la\tau-\rho\ra^{\gamma-3} \rho^{\gamma-\frac16}|\wh{h}(\rho)|d\rho\right\|_{L^2_\tau}.
\end{align*}
 
For $s>0$, we can bound the second term as follows
\begin{align*}
\left\|\int_0^\infty\la\tau-\rho\ra^{\gamma-3} \rho^{\gamma-\frac16}|\wh{h}(\rho)|d\rho\right\|_{L^2_\tau} &\lsim \left\|\int_0^\infty\la\tau-\rho\ra^{-\frac43} \la\rho\ra^{\frac{s+1}{3}}|\wh{h}(\rho)|d\rho\right\|_{L^2_\tau}\\
&\lsim \|\la \tau\ra^{-\frac43}\|_{L^2_\tau} \|\chi_{(0,\infty)} h\|_{H^{\frac{s+1}{3}}_t(\R)} \lsim \|h\|_{H^{\frac{s+1}{3}}_t(\R^+)}
\end{align*}
where we need $\gamma \leq \frac12+\frac{s}{3}$ so that $\gamma-\frac16 \leq \frac{s+1}{3}$. We also used the fact that $s<\frac72$, which implies $\gamma \leq \frac12+\frac{s}{3} < \frac53$ and thus $\gamma -3 <-\frac43$. 

\end{proof}

\section{Bilinear Estimates} \label{sec: bilinear} 
In order to close the fixed point argument, we must control the nonlinear terms, such as $\|F\|_{X^{s,-\frac12+}}$ and $\|F\|_{X^{\frac12+,\frac{s-2}{3}}}$ which arise in Proposition \ref{Duham_time_trace}. The choice of $-b$ and $\frac{2s-1}{6}-b = \frac{s-2}{3}+\left(\frac12-b\right)$ rather than $-\frac12+$ and $\frac{s-2}{3}$ in the estimates below is due to Lemma \ref{extractT}, so that we can extract a positive power of $T$. The detailed contraction argument is given later in Section \ref{subsec:local}. 

\begin{proposition} \label{prop:bil} \hfill \\
For $s>0$, $\max(\frac{3-s}{6},\frac{7}{16})<b<\frac12, \gamma>\frac12$,
\begin{align}
\|\del_x(v^2)\|_{X^{s,-b}} &\lsim \|v\|^2_{X^{s,b}_\alpha} \label{bil_1}\\
\| \del_x (uv) \|_{X^{s,-b}_\alpha} &\lsim \|u\|_{X^{s,b} \cap V^\gamma} \|v\|_{X^{s,b}_\alpha \cap V^\gamma} \label{bil_2}
\end{align}
For $\frac12<s<2$, $\max(\frac{s+1}{6},\frac{7}{16})<b<\frac12$, $\gamma > \frac12$,
\begin{align}
\|\del_x(v^2)\|_{X^{\frac12+,\frac{2s-1}{6}-b}} &\lsim \|v\|^2_{X^{s,b}_\alpha} \label{bil_3}\\
\| \del_x (uv) \|_{X^{\frac12+,\frac{2s-1}{6}-b}_\alpha} &\lsim \|u\|_{X^{s,b} \cap V^\gamma} \|v\|_{X^{s,b}_\alpha \cap V^{\gamma}} \label{bil_4}.
\end{align}
\end{proposition}

These estimates are the most technical ingredient in establishing well-posedness, as their proofs rely on the specific structure of our system \eqref{majdaIBVP}. We will also need several mixed Lebesgue norm estimates due to Kenig-Ponce-Vega, which we list in the following lemma. 

\begin{lemma} Lemmas 2.2, 2.4, \& 2.6 in \cite{Kenig93}.
\begin{equation}
\left \| \left[ \frac{\wh{f}(\xi,\tau)}{\la\tau\ra^{\frac12+}}\right]^\vee\right\|_{L^2_xL^\infty_t} \lsim \|\wh{f}\|_{L^2_{\xi,\tau}} \label{Sobolev}
\end{equation}

\begin{equation}
\left\|\left[\frac{|\xi|^{\theta}\wh{f}(\xi,\tau)}{\la\tau-\xi^3\ra^b}\right]^\vee\right\|_{L^4_xL^4_t} \lsim \| \wh{f}\|_{L^2_{\xi,\tau}} \quad \text{for }\, 0\leq \theta\leq \frac18, \; b>\frac38 \label{Strichartz_4}
\end{equation}

\begin{equation} \label{Kato_p}
\left\|\left[ \frac{|\xi|^{\frac{p-2}{p}}f(\xi,\tau)}{\la\tau-\xi^3\ra^{\frac{p-2}{2p}+}}\right]^\vee\right\|_{L^p_xL^2_t} \lsim \|f\|_{L^2_{\xi,\tau}} \qquad 2 < p < \infty.
\end{equation}
\end{lemma}
\noindent \eqref{Sobolev} follows directly from the Sobolev embedding inequality. \eqref{Strichartz_4} is established by interpolating between the $L^6L^6$ KdV Strichartz estimate and the Plancharel identity.  \eqref{Kato_p} follows from interpolation between the Kato smoothing inequality \eqref{Kato_smooth} and the Plancharel identity.

\begin{rmk} \label{rmk:Fourier_scaling}
 To apply these estimates when the denominator involves $\la\tau-\alpha \xi^3\ra$ instead of $\la\tau-\xi^3\ra$, we can use Fourier scaling, as in the following example:
\begin{align*}
\mathcal{F}^{-1}\left[\frac{\wh{f_2}(\xi,\tau)}{\la\tau-\alpha\xi^3\ra^b}\right](x,t) = \frac{1}{\alpha^{1/3}}\,\mathcal{F}^{-1}\left[\frac{\wh{f_2}(\frac{\xi}{\alpha^{1/3}},\tau)}{\la\tau-\xi^3\ra^b}\right]\left(\frac{x}{\alpha^{1/3}},t\right)
\end{align*}
Then by scaling in $L^4_x$ and (\ref{Strichartz_4}), we have
\begin{align*}
\left\|\left[\frac{\wh{f_2}(\xi,\tau)}{\la\tau-\alpha\xi^3\ra^b}\right]^\vee\right\|_{L^4_xL^4_t} &= \frac{\alpha^{1/12}}{\alpha^{1/3}} \left\|\left[\frac{\wh{f_2}(\frac{\xi}{\alpha^{1/3}},\tau)}{\la\tau-\xi^3\ra^b}\right]^\vee\right\|_{L^4_xL^4_t}\\
&\lsim \alpha^{-1/4}\left\|\wh{f_2}\left(\frac{\xi}{\alpha^{1/3}},\tau\right)\right\|_{L^2_{\xi,\tau}} = \alpha^{-1/12} \|\wh{f_2}\|_{L^2_{\xi,\tau}}.
\end{align*}
\end{rmk}

\subsection{Proof of \eqref{bil_1}}
\begin{proof}
We aim to show $$\|\del_x(v^2)\|_{X^{s,-b}} \lsim \|v\|^2_{X^{s,b}_\alpha}$$  by duality, so we define
\begin{equation} \label{duality} \int \limits_{\substack{\xi = \xi_1+\xi_2 \\ \tau = \tau_1+\tau_2}} \frac{|\xi| \la\xi\ra^s \wh{g}(\xi,\tau)}{\la\tau-\xi^3\ra^b} \frac{\la\xi_1\ra^{-s}\wh{f_1}(\xi_1,\tau_1)}{\la\tau_1-\alpha\xi_1^3\ra^b} \frac{\la\xi_2\ra^{-s}\wh{f_2}(\xi_2,\tau_2)}{\la\tau_2-\alpha\xi_2^3\ra^b} \, d\xi d\xi_1 d\tau d\tau_1\end{equation}

where 
\begin{align*}
\wh{g}(\xi,\tau) &= \la\xi\ra^{-s}\la\tau-\xi^3\ra^b|\wh{w}(\xi,\tau)| \qquad w \in X^{-s,b} \\
\wh{f_i}(\xi_i,\tau_i) &= \la\xi_i\ra^s\la\tau_i-\alpha\xi_i^3\ra^b |\wh{v}(\xi_i,\tau_i)| \qquad i=1,2.
\end{align*}

We remark that this choice of Fourier multipliers makes the integrand in \eqref{duality} real and non-negative. To prove \eqref{bil_1} it suffices to bound \eqref{duality} by $\|g\|_{L^2_{x,t}} \|f_1\|_{L^2_{x,t}} \|f_2\|_{L^2_{x,t}}$.

We begin with the approach from \cite{oh} used to establish the corresponding bilinear estimate on the full line, namely a standard Cauchy-Schwarz argument. Since we must take $b<\frac12$ on the half line (because of Proposition \ref{Duham_time_trace}) however, this turns out to be less successful. It gives the desired bound only in case where $|\xi|$ is small.\\

\textit{Case 1  \quad $|\xi| \lsim 1$}\\
By the Cauchy-Schwarz inequality in $\xi_1$ and $\tau_1$, and then Young's inequality (Lemma \ref{A.4}),  we can bound \eqref{duality} as follows:
\begin{align*}
\eqref{duality} &\leq \int \frac{|\xi| \la\xi\ra^s \wh{g} (\xi,\tau)}{\la\tau-\xi^3\ra^b} \left[\int \limits_{\substack{\xi = \xi_1+\xi_2 \\ \tau = \tau_1+\tau_2}} \frac{\la\xi_1\ra^{-2s} \la\xi_2\ra^{-2s}}{\la\tau_1-\alpha\xi_1^3\ra^{2b}\la\tau_2-\alpha\xi_2^3\ra^{2b}} d\xi_1 d\tau_1 \right]^{\frac12} \left(|\wh{f_1}|^2 * |\wh{f_2}|^2\right)^{\frac12} d\xi d\tau \\
& \leq \left[\sup_{\xi,\tau} \frac{|\xi|^2 \la\xi\ra^{2s}}{\la\tau-\xi^3\ra^{2b}} \int \limits_{\substack{\xi = \xi_1+\xi_2 \\ \tau = \tau_1+\tau_2}} \frac{\la\xi_1\ra^{-2s} \la\xi_2\ra^{-2s}}{\la\tau_1-\alpha\xi_1^3\ra^{2b}\la\tau_2-\alpha\xi_2^3\ra^{2b}} d\xi_1 d\tau_1 \right]^{\frac12} \int |\wh{g}(\xi,\tau)| \left(|\wh{f_1}|^2 * |\wh{f_2}|^2\right)^{\frac12} d\xi d\tau\\
& \leq \left[\sup_{\xi,\tau} \frac{|\xi|^2 \la\xi\ra^{2s}}{\la\tau-\xi^3\ra^{2b}} \int \limits_{\substack{\xi = \xi_1+\xi_2 \\ \tau = \tau_1+\tau_2}} \frac{\la\xi_1\ra^{-2s} \la\xi_2\ra^{-2s}}{\la\tau_1-\alpha\xi_1^3\ra^{2b}\la\tau_2-\alpha\xi_2^3\ra^{2b}} d\xi_1 d\tau_1 \right]^{\frac12} \|\wh{g}\|_{L^2_{\xi,\tau}} \, \||\wh{f_1}|^2 * |\wh{f_2}|^2 \|_{L^1_{\xi,\tau}}^{\frac12}\\
& \leq \left[\sup_{\xi,\tau} \frac{|\xi|^2 \la\xi\ra^{2s}}{\la\tau-\xi^3\ra^{2b}} \int \limits_{\substack{\xi = \xi_1+\xi_2 \\ \tau = \tau_1+\tau_2}} \frac{\la\xi_1\ra^{-2s} \la\xi_2\ra^{-2s}}{\la\tau_1-\alpha\xi_1^3\ra^{2b}\la\tau_2-\alpha\xi_2^3\ra^{2b}} d\xi_1 d\tau_1 \right]^{\frac12} \|\wh{g}\|_{L^2_{\xi,\tau}} \|\wh{f_1}\|_{L^2_{\xi,\tau}} \|\wh{f_2}\|_{L^2_{\xi,\tau}}.
\end{align*}

Then by the Plancherel identity, it suffices to bound the supremum:
\begin{equation} \label{dual_M}
\sup_{\xi,\tau} \frac{|\xi|^2 \la\xi\ra^{2s}}{\la\tau-\xi^3\ra^{2b}} \int \limits_{\substack{\xi = \xi_1+\xi_2 \\ \tau = \tau_1+\tau_2}} \frac{\la\xi_1\ra^{-2s} \la\xi_2\ra^{-2s}}{\la\tau_1-\alpha\xi_1^3\ra^{2b}\la\tau_2-\alpha\xi_2^3\ra^{2b}} d\xi_1 d\tau_1
\end{equation}

by a constant.
 
Since $\la\xi\ra = \la\xi_1+\xi_2\ra \lsim \la \xi_1\ra \la \xi_2\ra$, we have
\begin{align*}
\eqref{dual_M} &= \sup_{\xi,\tau} \frac{|\xi|^2}{\la\tau-\xi^3\ra^{2b}} \int \limits_{\xi =\xi_1+\xi_2} \frac{\la\xi\ra^{2s}}{\la\xi_1\ra^{2s}\la\xi_2\ra^{2s}}\frac{d\tau_1 d\xi_1}{\la\tau_1-\alpha \xi_1^3\ra^{2b} \la\tau-\tau_1-\alpha\xi_2^3\ra^{2b}} \\
& \lsim \sup_{\xi,\tau} \frac{|\xi|^2}{\la\tau-\xi^3\ra^{2b}} \int \limits_{\xi = \xi_1+\xi_2}
\frac{d\tau_1 d\xi_1}{\la\tau_1 -\alpha \xi_1^3\ra^{2b} \la \tau-\tau_1-\alpha\xi_2^3\ra^{2b}} \\
&\lsim \sup_{\xi,\tau} \frac{|\xi|^2}{\la\tau-\xi^3\ra^{2b}} \int \frac{d\xi_1}{\la \tau-\alpha\xi_1^3 - \alpha(\xi-\xi_1)^3\ra^{4b-1}}
\end{align*}
where we've used Lemma \ref{A.2} in the last line. Now, by a change of variables
\begin{align*}
\eta &= \tau-\alpha\xi_1^3 - \alpha(\xi-\xi_1)^3 = -3\alpha\xi \xi_1^2 +3 \alpha \xi^2 \xi_1 +\tau-\alpha \xi^3\\
\xi_1 &= \frac{-3\alpha\xi^2 \pm \sqrt{(3\alpha \xi)(4\tau-4\eta-\alpha\xi^3)}}{-6\alpha\xi} \\
d\eta &= (-6\alpha\xi\xi_1+3\alpha\xi^2) d\xi_1 = \pm \sqrt{(3\alpha\xi)(4\tau-4\eta-\alpha\xi^3)} d\xi_1.
\end{align*}
Then Lemma \ref{A.3} (with $\frac38<b<\frac12$) gives the desired bound:
\begin{align*}
\eqref{dual_M} &\lsim \sup_{\xi,\tau} \frac{|\xi|^2}{\la\tau-\xi^3\ra^{2b}} \int \frac{d\eta}{\sqrt{3\alpha |\xi|} \la\eta\ra^{4b-1}\sqrt{|\eta-\tau + \frac{\alpha}{4}\xi^3|}} \\
& \lsim \sup_{\xi,\tau} \frac{|\xi|^\frac32}{\la\tau-\xi^3\ra^{2b}\la\tau-\frac{\alpha}{4}\xi^3\ra^{4b-\frac32}} \lsim 1 \qquad (\text{since } |\xi|\lsim 1 \text{ in Case 1)}.
\end{align*}

If we could take $b>\frac12$ in the last line, the supremum would be bounded regardless of the size of $|\xi|$ since we always have either $|\tau-\xi^3| \gsim |\xi|^3$ or $|\tau-\frac{\alpha}{4}\xi^3|\gsim |\xi|^3$. This is how the proof concludes in \cite{oh} for the bilinear estimate on the full line. However, we need $b<\frac12$ for the IBVP on $\R^+$, so we must consider further cases.\\

For Case 2 ($|\xi| \gg 1$), the following resonance calculation will be useful
\begin{equation}
\tau-\xi^3 - \tau_1+\alpha \xi_1^3 -\tau_2 + \alpha \xi_2^3 =  (\alpha-1)\xi (\xi - r_1 \xi_1) (\xi - r_2 \xi_1) \label{max_id}
\end{equation}
where 
$$r_1 = \frac{3\alpha - \sqrt{3 \alpha (4-\alpha)}}{2(\alpha-1)},  \quad
r_2 = \frac{3\alpha + \sqrt{3 \alpha (4-\alpha)}}{2(\alpha-1)}.$$

Resonances occur when \eqref{max_id} is small. This suggests we consider the cases $|\xi|\lsim 1$, $\xi \sim r_1 \xi_1$, and $\xi \sim r_2 \xi_1$. We have already established the bound for the first case above. We now define these remaining cases more explicitly. Since $|\frac{r_2}{r_1}|>1$  (and $0<r_1<1$ for $0<\alpha<1$), we may choose $1 < c <\sqrt{|\frac{r_2}{r_1}|}$. 
Then define
$$A := \{\xi,\xi_1: c^{-1}|r_1 \xi_1|<|\xi|<c|r_1 \xi_1|\} \text{ and } 
B := \{\xi,\xi_1: c^{-1}|r_2 \xi_1|<|\xi|<c|r_2 \xi_1|\}.$$

Our choice of $c$ ensures these sets are disjoint. On $A^c$, we have $|\xi-r_1 \xi_1| \geq (1-c^{-1})|\xi|$. And likewise on $B^c$, we have $|\xi-r_2\xi_1|\geq (1-c^{-1})|\xi|$. This leads us to three subcases: $A$, $B$, and $C = A^c \cap B^c$.\\

\textit{Case 2A} ($|\xi| \gg 1$ and $\xi,\xi_1 \in A$)\\
For this subcase, we bound $\eqref{duality}$ as in Case 1 using Cauchy-Schwarz, but in $\xi$ and $\tau$ this time. It therefore suffices to bound

\begin{equation} \label{dual_M1}
\sup_{\xi_1,\tau_1} \frac{\la\xi_1\ra^{-2s}}{\la\tau_1-\alpha\xi_1^3\ra^{2b}} \int \limits_{\substack{\xi=\xi_1+\xi_2 \\ \tau= \tau_1+\tau_2}} \frac{|\xi|^2\la\xi\ra^{2s}}{\la\xi_2\ra^{2s}} \frac{1}{\la\tau-\xi^3\ra^{2b} \la\tau_2-\alpha\xi_2^3\ra^{2b}}d\tau d\xi
\end{equation}
by an absolute constant.\\

Using Lemma \ref{A.2}, \eqref{max_id}, and the fact that $|\xi|\lsim |\xi_1|$ on $A$, we have
\begin{align*}
\eqref{dual_M1} &\stackrel{\phantom{\la a \ra\la b\ra\gsim \la a-b \ra}}{\lsim} \sup_{\xi_1,\tau_1} \int \limits_{\xi = \xi_1+\xi_2} \frac{|\xi|^2}{\la\xi_2\ra^{2s}} \frac{d\xi}{\la\tau_1-\alpha\xi_1^3\ra^{2b} \la\tau_1 - \alpha \xi_1^3 + (\alpha-1)\xi(\xi-r_1\xi_1)(\xi-r_2\xi_1)\ra^{4b-1}} \\
&\stackrel{\la a \ra\la b\ra\gsim \la a-b \ra}{\lsim} \sup_{\xi_1,\tau_1}\int \frac{|\xi|^2}{\la\xi-\xi_1\ra^{2s}}\frac{d\xi}{ \la\xi(\xi-r_1\xi_1)(\xi-r_2\xi_1)\ra^{4b-1}} \\
&\stackrel{\phantom{\la a \ra\la b\ra\gsim \la a-b \ra}}{\lsim} \sup_{\xi_1,\tau_1}\int \frac{|\xi|^2}{\la\xi-\xi_1\ra^{2s}}\frac{d\xi}{\la\xi^2(\xi-r_1\xi_1)\ra^{4b-1}}.
\end{align*}

From here, we separate cases based on the size of $|\xi-r_1\xi_1|$, keeping in mind that $0<r_1<1$. We have
\begin{align*}
\int \limits_{|\xi-r_1\xi_1|>1} \frac{|\xi|^2}{\la\xi-\xi_1\ra^{2s}}\frac{d\xi}{\la\xi^2(\xi-r_1\xi_1)\ra^{4b-1}} &\lsim \int \limits_{|\xi-r_1\xi_1|>1} \frac{|\xi|^2}{\la\xi-\xi_1\ra^{2s} \la\xi\ra^{8b-2}\la\xi-r_1\xi_1\ra^{4b-1}}d\xi \\
&\lsim|\xi_1|^{8(\frac12-b)}\int\frac{d\xi}{\la\xi-\xi_1\ra^{2s}\la\xi-r_1\xi_1\ra^{4b-1}} \\
&\lsim |\xi_1|^{8(\frac12-b)}\frac{1}{\la(1-r_1)\xi_1\ra^{2\min{(s,1/2)}+4b-2}} \lsim 1
\end{align*}
where we need $6b>3-s$ for the last inequality if $s<\frac12$, and $b>\frac{5}{12}$ if $s>\frac12$. To apply Lemma \ref{A.2}, we needed $2b>1-s$, but this is weaker than the requirement that $6b>3-s$. \\

For the subcase with $|\xi-r_1\xi_1|\leq 1$, since $|\xi|\lsim |\xi_1|$ on $A$ we have
\begin{align*}
\int \limits_{|\xi-r_1\xi_1|\leq1} \frac{|\xi|^2}{\la\xi-\xi_1\ra^{2s}}\frac{d\xi}{\la\xi^2(\xi-r_1\xi_1)\ra^{4b-1}} &\lsim \int \limits_{|\xi-r_1\xi_1|\leq1}\frac{|\xi|^2}{\la(1-r_1)\xi_1\ra^{2s} |\xi|^{8b-2}|\xi-r_1\xi_1|^{4b-1}} d\xi \\
&\lsim \frac{|\xi_1|^{4-8b}}{\la(1-r_1)\xi_1\ra^{2s}} \int \limits_{|\xi-r_1\xi_1|\leq1} \frac{d\xi}{|\xi-r_1\xi_1|^{4b-1}} \lsim 1.
\end{align*}
We need $s>2-4b$ for the last inequality, but again this follows from $s>3-6b$.\\

\textit{Case 2B} ($|\xi| \gg 1$ and $\xi,\xi_1 \in B$)\\
For $0<\alpha<1$ we have $r_2 < 0$. In particular $r_2 \neq 1$, so the argument from Case 2A carries over, but with the roles of $r_1$ and $r_2$ reversed.\\

\textit{Case 2C} ($|\xi| \gg 1$ and $\xi,\xi_1 \in C = A^c \cap B^c$)\\
We separate into further subcases depending on which factor dominates in the right side of \eqref{max_id}. Define
$$M:= \max{\left( |\tau-\xi^3|, |\tau_1 - \alpha \xi_1^3|, |\tau_2-\alpha\xi_2^3| \right)}.$$
Then by \eqref{max_id}, together with the inequalities $|\xi-r_1\xi_1|\gsim|\xi|$ and $|\xi - r_2 \xi_1|\gsim |\xi|$ (which hold for $\xi,\xi_1 \in C$), we have 
\begin{equation} \label{max_lb}
M \geq \frac{|\alpha-1|}{3}|\xi||\xi-r_1\xi_1||\xi-r_2\xi_1| \gsim |\xi|^3.
\end{equation}

If $M = |\tau -\xi^3|$, we have
\begin{align*}
\eqref{duality} &\lsim \int \limits_{\substack{\xi=\xi_1+\xi_2 \\ \tau= \tau_1+\tau_2}} \frac{\la\xi\ra^{s}}{\la\xi_1\ra^s\la\xi_2\ra^s} \frac{|\xi| \wh{g}(\xi,\tau)}{\la\xi\ra^{3b}} \frac{\wh{f_1}(\xi_1,\tau_1)}{\la\tau_1 -\alpha\xi_1^3\ra^b}\frac{\wh{f_2}(\xi_2,\tau_2)}{\la\tau_2 -\alpha\xi_2^3\ra^b} \, d\xi d\xi_1 d\tau d\tau_1 \\
&\lsim \int \limits_{\substack{\xi=\xi_1+\xi_2 \\ \tau= \tau_1+\tau_2}}\wh{g}(\xi,\tau) \frac{\wh{f_1}(\xi_1,\tau_1)}{\la\tau_1 -\alpha\xi_1^3\ra^b}\frac{\wh{f_2}(\xi_2,\tau_2)}{\la\tau_2 -\alpha\xi_2^3\ra^b} \, d\xi_1 d\tau_1 \;d\xi d\tau\\
&= \int g(x,t) \, \left[\frac{\wh{f_1}}{\la\tau-\alpha \xi^3\ra^b} \right]^\vee \, \left[\frac{\wh{f_2}}{\la\tau-\alpha \xi^3\ra^b} \right]^\vee \, dx dt\\
&\lsim \|g\|_{L^2_{x,t}} \left\|\left[\frac{\wh{f_1}}{\la\tau -\alpha\xi^3\ra^b}\right]^\vee\right\|_{L^4_xL^4_t} \left\|\left[\frac{\wh{f_2}}{\la\tau -\alpha\xi^3\ra^b}\right]^\vee\right\|_{L^4_xL^4_t} 
\end{align*}
where we've used $s\geq 0$, $b>\frac13$, $\la\xi\ra \lsim \la\xi_1\ra\la\xi_2\ra$, then Parseval's identity and H\"older's inequality. Finally, since $b>\frac38$, \eqref{Strichartz_4} gives the desired bound. \\

If $M = |\tau_1-\alpha \xi_1^3|$ or $M = |\tau_2-\alpha\xi_2^3|$, the argument is virtually identical. We may replace the $\la M\ra$ with $\la \xi\ra^3$ in the denominator and cancel. Then we apply H\"older's inequality, using (\ref{Strichartz_4}) for the two remaining factors with denominators. 
\end{proof}

\subsection{Proof of \eqref{bil_3}}

We postpone the proof of (\ref{bil_2}) for the moment and prove (\ref{bil_3}) next because it has the same resonances as (\ref{bil_1}).

\begin{proof}
We must show
$$\|\del_x(v^2)\|_{X^{\frac12+,\frac{2s-1}{6}-b}} \lsim \|v\|^2_{X^{s,b}_\alpha}.$$ 
Arguing by duality as before, we must now bound the quantity

\begin{equation} \label{duality*}
\int \limits_{\substack{\xi = \xi_1+\xi_2 \\ \tau = \tau_1+\tau_2}} \frac{|\xi| \la\xi\ra^{\frac12+} \wh{g}(\xi,\tau)}{\la\tau-\xi^3\ra^{b-\frac{2s-1}{6}}} \frac{\la\xi_1\ra^{-s}\wh{f_1}(\xi_1,\tau_1)}{\la\tau_1-\alpha\xi_1^3\ra^b} \frac{\la\xi_2\ra^{-s}\wh{f_2}(\xi_2,\tau_2)}{\la\tau_2-\alpha\xi_2^3\ra^b} \, d\xi d\xi_1 d\tau d\tau_1 
\end{equation}
by  $\|g\|_{L^2} \|f_1\|_{L^2} \|f_2\|_{L^2}$. Note that for $s<2$, we have $b-\frac{2s-1}{6}>b-\frac{s+1}{6}$, which is assumed to be positive.

\textit{Case 1} ($|\xi|\lsim 1$)\\
The Cauchy-Schwarz argument as in the proof of (\ref{bil_1}) carries over:
\begin{align*}
\eqref{duality*} &:= \sup_{\xi,\tau} \frac{|\xi|^2}{\la\tau-\xi^3\ra^{2b-\frac{2s-1}{3}}}
	\int_{\xi=\xi_1+\xi_2} \frac{\la\xi\ra^{1+}}{\la\xi_1\ra^{2s}\la
	\xi_2\ra^{2s}} \frac{d\tau_1 d\xi_1}{\la\tau_1-\alpha\xi_1^3\ra^{2b}
	\la\tau-\tau_1-\alpha\xi_2^3\ra^{2b}} \\
&\lsim \sup_{\xi,\tau} \frac{|\xi|^2}{\la\tau-\xi^3\ra^{2b-\frac{2s-1}{3}}}
	\int_{\xi=\xi_1+\xi_2} \frac{ d\xi_1}
	{\la\tau-\alpha\xi_1^3-\alpha\xi_2^3\ra^{4b-1}}\\
&\lsim \sup_{\xi,\tau} \frac{1}{\la\tau-\xi^3\ra^{2b-\frac{2s-1}{3}}\la
	\tau-\frac{\alpha}{4}\xi^3\ra^{4b-\frac32}} \lsim 1
\end{align*}
for $\frac12+<s<2$, $b>\frac38$, and $b>\frac{2s-1}{6}$. In the last line we've used the change of variables and Lemma \ref{A.3} just as in Case 1 of the proof of (\ref{bil_1}).\\

For Case 2 ($|\xi|\gg1$), we use the same subcases as in the proof of \eqref{bil_1}, again based on the sets $$ A = \{\xi,\xi_1: c^{-1}|r_1 \xi_1|<|\xi|<c|r_1 \xi_1|\} \text{ and } 
B = \{\xi,\xi_1: c^{-1}|r_2 \xi_1|<|\xi|<c|r_2 \xi_1|\}.$$

\textit{Case 2A}($|\xi|\gg1$, and $\xi,\xi_1 \in A$)\\
Following the proof of (\ref{bil_1}), we apply Cauchy-Schwarz in $\xi$ and $\tau$. It therefore suffices to bound 
\begin{equation} \label{CS3}
 \sup_{\xi_1,\tau_1} \frac{1}{\la\xi_1\ra^{2s}\la\tau_1-\alpha\xi_1^3\ra^{2b}} \int_{\xi = \xi_1+\xi_2} \frac{|\xi|^2\la\xi\ra^{1+}d\tau \,d\xi}{\la\xi_2\ra^{2s}\la\tau-\xi^3\ra^{2b-\frac{2s-1}{3}}\la\tau-\tau_1-\alpha\xi_2^3\ra^{2b}} 
\end{equation}
by a constant. Since $|\xi|\sim |\xi_1|$ on $A$, we have
\begin{align*}
\eqref{CS3}&\lsim \sup_{\xi_1,\tau_1} \frac{\la\xi_1\ra^{3+}}{\la\xi_1\ra^{2s}\la\tau_1-\alpha\xi_1^3\ra^{2b}} \int \frac{d\xi}{\la\xi-\xi_1\ra^{2s} \la\tau_1-\alpha\xi_1^3+(\alpha-1)\xi(\xi-r_1\xi_1)(\xi-r_2\xi_1)\ra^{4b-\frac{2s+2}{3}}}\\
&\lsim \sup_{\xi_1,\tau_1} \frac{\la\xi_1\ra^{3+}}{\la\xi_1\ra^{2s}\la\tau_1-\alpha\xi_1^3\ra^{\frac{2s+2}{3}-2b}} \int \frac{d\xi}{\la\xi-\xi_1\ra^{2s}\la\xi^2(\xi-r_1\xi_1)\ra^{4b-\frac{2s+2}{3}}}\\
&\lsim \sup_{\xi_1}\,  \la\xi_1\ra^{3-2s+} \int \frac{d\xi}{\la\xi-\xi_1\ra^{2s}\la\xi^2(\xi-r_1\xi_1)\ra^{4b-\frac{2s+2}{3}}} \qquad \text{(since }s>\frac12>b)
\end{align*}
where we needed $b>\frac{s+1}{6}$ to apply Lemma \ref{A.2} in the first line. We now consider separate subcases for $|\xi-r_1\xi_1|>1$ and $|\xi-r_1\xi_1|\leq1$. For the first subcase,
\begin{align*}
&\sup_{\xi_1}\,  \la\xi_1\ra^{3-2s+} \int_{|\xi-r_1\xi_1|>1} 
	\frac{d\xi}{\la\xi-\xi_1\ra^{2s}(\la\xi^2\ra \la
	\xi-r_1\xi_1\ra)^{4b-\frac{2s+2}{3}}}\\
\lsim &\sup_{\xi_1} \, \la\xi_1\ra^{\frac{13-2s}{3}-8b+}
	\int \frac{d\xi}{\la\xi-\xi_1\ra^{2s}
	\la\xi-r_1\xi_1\ra^{4b-\frac{2s+2}{3}}}\\
\lsim &\sup_{\xi_1} \, \frac{\la\xi_1\ra^{\frac{13-2s}{3}-8b+}}{
	\la(1-r_1)\xi_1\ra^{4b-\frac{2s+2}{3}}}
\lsim \sup_{\xi_1} \, \la\xi_1\ra^{5-12b+} \lsim 1
\end{align*}
for $b>\frac{5}{12}$. \\

For the other subcase where $|\xi-r_1\xi_1|\leq1$
\begin{align*}
&\sup_{\xi_1} \, \la\xi_1\ra^{3-2s+} \int_{|\xi-r_1\xi_1|\leq1} \frac{d\xi}{\la\xi-\xi_1\ra^{2s}|\xi^2(\xi-r_1\xi_1)|^{4b-\frac{2s+2}{3}}}\\
\sim &\sup_{\xi_1}\, \frac{\la\xi_1\ra^{3-2s+}}{\la(1-r_1)\xi_1\ra^{2s}} \int_{|\xi-r_1\xi_1|\leq1} \frac{d\xi}{ (\la\xi\ra^2|\xi-r_1\xi_1|)^{4b-\frac{2s+2}{3}}}\\
\sim &\sup_{\xi_1} \, \la\xi_1\ra^{\frac13(13-8s-24b)+}\int_{|\xi-r_1\xi_1|\leq1} \frac{d\xi}{|\xi-r_1\xi_1|^{4b-\frac{2s+2}{3}}}\\
\lsim &\sup_{\xi_1} \, \la\xi_1\ra^{\frac13(13-8s-24b)+} \lsim 1
\end{align*}
for $b>\frac{13-8s}{24}$. For $\frac12<s<2$, this is weaker than the $b>\frac{5}{12}$ requirement of the previous subcase. We also used the fact that $s>\frac12>b>\frac{s+1}{6}$ implies $0\leq 4b-\frac{2s+2}{3} <1$.

\textit{Case 2B} ($\xi\gg1$, and $\xi$,$\xi_1 \in B$)\\
The same argument applies with the roles of $r_1$ and $r_2$ interchanged. Recall $r_2 \neq 1$ for $0<\alpha <1$. \\

\textit{Case 2C} ($\xi\gg1$, and $\xi,\xi_1 \in A^c \cap B^c$)\\
Just as in the proof of (\ref{bil_1}), we use (\ref{max_lb}) and consider separate subcases based on the maximum denominator $M$.\\

If $M = |\tau-\xi^3| \gsim |\xi|^3$, 
\begin{align*}
\eqref{duality*} &\lsim\quad \int\limits_{\substack{\xi = \xi_1+\xi_2 \\ \tau = \tau_1+\tau_2}} \frac{|\xi|\la\xi\ra^{\frac12+}\wh{g}(\xi,\tau)}{\la\xi\ra^{\frac12-s+3b}} \frac{\la\xi_1\ra^{-s}\wh{f_1}(\xi_1,\tau_1)}{\la\tau_1-\alpha\xi_1^3\ra^b} \frac{\la\xi_2\ra^{-s}\wh{f_2}(\xi_2,\tau_2)}{\la\tau_2-\alpha\xi_2^3\ra^b}d\xi d\xi_1 d\tau d\tau_1\\
&\lsim \int\limits_{\substack{\xi = \xi_1+\xi_2 \\ \tau = \tau_1+\tau_2}} \wh{g} \frac{\la\xi_1\ra^{1+s-3b+}\wh{f_1}}{\la\xi_1\ra^s\la\tau_1-\alpha\xi_1^3\ra^b} \frac{\la\xi_2\ra^{1+s-3b+}\wh{f_2}}{\la\xi_2\ra^s\la\tau_2-\alpha\xi_2^3\ra^b} d\xi d\xi_1 d\tau d\tau_1 \qquad \text{(since } \la \xi \ra \lsim \la\xi_1\ra \la \xi_2\ra)\\
&\lsim \int\limits_{\substack{\xi = \xi_1+\xi_2 \\ \tau = \tau_1+\tau_2}} \wh{g} \frac{\wh{f_1}}{\la\tau_1-\alpha\xi_1^3\ra^b} \frac{\wh{f_2}}{\la\tau_2-\alpha\xi_2^3\ra^b} d\xi d\xi_1 d\tau d\tau_1\\
&\lsim \|\wh{g}\|_{L^2_{\xi,\tau}} \left\|\left[\frac{\wh{f_1}}{\la\tau -\alpha\xi^3\ra^b}\right]^\vee\right\|_{L^4_xL^4_t} \left\|\left[\frac{\wh{f_2}}{\la\tau -\alpha\xi^3\ra^b}\right]^\vee\right\|_{L^4_xL^4_t}\\
& \lsim \|g\|_{L^2} \|f_1\|_{L^2} \|f_2\|_{L^2}
\end{align*}
where we need $b>\frac13$ for the third line, and then $b>\frac38$ to use (\ref{Strichartz_4}) in the last line.

If $M = |\tau_1-\alpha\xi_1^3|\gsim |\xi|^3$, we start with $M\geq |\tau-\xi^3|$ to adjust the denominators
$$\la\tau_1-\alpha\xi_1^3\ra^b  = \la\tau_1-\alpha\xi_1^3\ra^{b-\frac{2s-1}{6}}\la\tau_1-\alpha\xi_1^3\ra^{\frac{2s-1}{6}} \geq \la\tau_1-\alpha\xi_1^3\ra^{b-\frac{2s-1}{6}}\la\tau-\xi^3\ra^{\frac{2s-1}{6}}.$$
After this tradeoff, we apply (\ref{max_lb}) and proceed exactly as above
\begin{align*}
\eqref{duality*} &=\int \limits_{\substack{\xi = \xi_1+\xi_2 \\ \tau = \tau_1+\tau_2}} \frac{|\xi|\la\xi\ra^{\frac12+}}{\la\xi_1\ra^s\la\xi_2\ra^s} \frac{\wh{g}(\xi,\tau)}{\la\tau-\xi^3\ra^{b-\frac{2s-1}{6}}} \frac{\wh{f_1}(\xi_1,\tau_1)}{\la\tau_1-\alpha\xi_1^3\ra^b} \frac{\wh{f_2}(\xi_2,\tau_2)}{\la\tau_2-\alpha\xi_2^3\ra^b} d\xi d\xi_1 d\tau d\tau_1\\
&\lsim \int \limits_{\substack{\xi = \xi_1+\xi_2 \\ \tau = \tau_1+\tau_2}} \la\xi\ra^{\frac32-s+} \frac{\wh{g}}{\la\tau-\xi^3\ra^b} \frac{\wh{f_1}}{\la\tau_1-\alpha\xi_1^3\ra^{b-\frac{2s-1}{6}}} \frac{\wh{f_2}}{\la\tau_2-\alpha\xi_2\ra^b} d\xi d\xi_1 d\tau d\tau_1\\
&\lsim \int \limits_{\substack{\xi = \xi_1+\xi_2 \\ \tau = \tau_1+\tau_2}} \la\xi\ra^{\frac32-s+} \frac{\wh{g}}{\la\tau-\xi^3\ra^b} \frac{\wh{f_1}}{\la\xi\ra^{\frac12-s+3b}} \frac{\wh{f_2}}{\la\tau_2-\alpha\xi_2\ra^b} d\xi d\xi_1 d\tau d\tau_1\\
& \lsim \int\limits_{\substack{\xi = \xi_1+\xi_2 \\ \tau = \tau_1+\tau_2}}\frac{\wh{g}}{\la\tau-\xi^3\ra^b}\; \wh{f_1}\; \frac{\wh{f_2}}{\la\tau_2-\alpha\xi_2^3\ra^b} d\xi d\xi_1 d\tau d\tau_1\\
&\lsim  \left\|\left[\frac{\wh{g}}{\la\tau -\xi^3\ra^b}\right]^\vee\right\|_{L^4_xL^4_t} \|\wh{f_1}\|_{L^2_{\xi,\tau}}\left\|\left[\frac{\wh{f_2}}{\la\tau -\alpha\xi^3\ra^b}\right]^\vee\right\|_{L^4_xL^4_t}\\
& \lsim \|g\|_{L^2_{\xi,\tau}} \|f_1\|_{L^2_{\xi,\tau}} \|f_2\|_{L^2_{\xi,\tau}}.
\end{align*}

Finally, if $M = |\tau_2-\alpha\xi_2^3|$ the argument is identical, but with the 1 and 2 subscripts interchanged.  
\end{proof}
\subsection{Proof of \eqref{bil_2}}
\begin{proof}
We must show
$$\| \del_x (uv) \|_{X^{s,-b}_\alpha} \lsim \|u\|_{X^{s,b} \cap V^\gamma} \|v\|_{X^{s,b}_\alpha \cap V^\gamma}.$$

As with (\ref{bil_1}), we argue by duality. We must bound
\begin{equation} \label{duality**}
 \int \limits_{\substack{\xi = \xi_1+\xi_2 \\ \tau = \tau_1+\tau_2}} \frac{|\xi|\la\xi\ra^s\wh{g}(\xi,\tau)}{\la\tau-\alpha\xi^3\ra^b} \frac{\wh{f_1}(\xi_1,\tau_1)}{\beta_1(\xi_1,\tau_1)} \frac{\wh{f_2}(\xi_2,\tau_2)}{\beta_\alpha(\xi_2,\tau_2)} d\xi d\xi_1 d\tau d\tau_1\end{equation}
by $\|g\|_{L^2_{x,t}} \|f_1\|_{L^2_{x,t}} \|f_2\|_{L^2_{x,t}}$, where 
\begin{align*}
\beta_\alpha (\xi_i,\tau_i) &= \la\xi_i\ra^s \la \tau_i - \alpha \xi_i^3\ra^b + \chi_{|\xi_i|\leq 1} \la \tau_i\ra^{\gamma}.\\
\end{align*}

The argument is similar to the proof of (\ref{bil_1}), but not identical due to the different resonant cases and low-frequency considerations. Before exploring the resonant cases however, we can immediately dispense with the easiest case where $|\xi|\lsim 1$ using \eqref{Strichartz_4}. 

\textit{Case 1 ($|\xi|\lsim 1$)}
\begin{align*}
\eqref{duality**} &\lsim \int \limits_{\substack{\xi = \xi_1+\xi_2 \\ \tau = \tau_1+\tau_2}} \wh{g}(\xi,\tau) \frac{\wh{f_1}(\xi_1,\tau_1)}{\la\tau_1-\xi_1^3\ra^b} \frac{\wh{f_2}(\xi_2,\tau_2)}{\la\tau_2-\alpha \xi_2^3\ra^b} d\xi d\xi_1 d\tau d\tau_1\\
&\lsim \|\wh{g}\|_{L^2_{\xi,\tau}} \left\|\left[\frac{\wh{f_1}}{\la\tau -\xi^3\ra^b}\right]^\vee\right\|_{L^4_xL^4_t} \left\|\left[\frac{\wh{f_2}}{\la\tau -\alpha\xi^3\ra^b}\right]^\vee\right\|_{L^4_xL^4_t}\\
& \lsim \|g\|_{L^2_{x,t}} \|f_1\|_{L^2_{x,t}} \|f_2\|_{L^2_{x,t}}.
\end{align*}

We now turn to the resonant cases. Instead of (\ref{max_id}), we have the algebraic relation
\begin{equation}
\tau_1-\xi_1^3 +\tau_2-\alpha\xi_2^3-\tau+\alpha\xi^3 = (\alpha-1) r_1 r_2 \xi_1 \left(\xi-\frac{1}{r_1} \xi_1\right) \left(\xi-\frac{1}{r_2}\xi_1\right) \label{max_id2}
\end{equation}
and thus
\begin{equation}
M_2 := \max{(|\tau-\alpha\xi^3|, |\tau_1-\xi_1^3|, |\tau_2-\alpha\xi_2^3|)} \gsim |\xi_1|\left|\xi-\frac{1}{r_2}\xi_1\right|\left|\xi-\frac{1}{r_1}\xi_1\right| \label{max_lb2}
\end{equation}
where $r_1$ and $r_2$ are the same roots as in (\ref{max_id})
$$r_1 = \frac{3\alpha - \sqrt{3 \alpha (4-\alpha)}}{2(\alpha-1)},  \quad
r_2 = \frac{3\alpha + \sqrt{3 \alpha (4-\alpha)}}{2(\alpha-1)}.$$

\eqref{max_lb2} leads us to consider the case where $|\xi_1|\lsim 1$, which was not necessary in the earlier proofs of \eqref{bil_1} and \eqref{bil_3}. We may also define our remaining cases by choosing $1<c<\sqrt{|\frac{r_2}{r_1}|}$ and taking
$$A' := \{\xi,\xi_1: c^{-1}|r_1^{-1}\xi_1|<|\xi|<c|r_1^{-1}\xi_1|\} \text{ and } 
B' := \{\xi,\xi_1: c^{-1}|r_2^{-1}\xi_1|<|\xi|<c|r_2^{-1}\xi_1|\}$$
so that these sets are disjoint, and we have $|\xi-r_1^{-1}\xi_1|\gsim |\xi|$ on $(A')^c$ and $|\xi-r_2^{-1}\xi_1|\gsim |\xi|$ on $(B')^c$. Let $C':= (A')^c \cap (B')^c$. \\

\textit{Case 2 ($|\xi_1|\lsim 1 \ll |\xi|$)}\\
Here we require the low frequency modification introduced by Kenig and Colliander for the KdV equation on the half line
$$\|u\|_{V^\gamma} = \|\la \tau\ra^\gamma \chi_{(-1,1)}(\xi)\wh{u}
	(\xi,\tau) \|_{L^2_{\xi,\tau}} \qquad \gamma > \frac12.$$
In this case, since $\la\xi\ra = \la\xi_1+\xi_2\ra \sim \la\xi_2\ra$ we have
\begin{align*}
\eqref{duality**} &\lsim \int \limits_{\substack{\xi = \xi_1+\xi_2 \\ \tau = \tau_1+\tau_2}} \frac{|\xi|\la\xi\ra^s\wh{g}(\xi,\tau)}{\la\tau-\alpha\xi^3\ra^b} \frac{\wh{f_1}(\xi_1,\tau_1)}{\la\tau_1\ra^\gamma} \frac{\wh{f_2}(\xi_2,\tau_2)}{\la\xi_2\ra^s \la\tau_2-\alpha\xi_2^3\ra^b} d\xi d\xi_1 d\tau d\tau_1 \\
&\lsim \int \limits_{\substack{\xi = \xi_1+\xi_2 \\ \tau = \tau_1+\tau_2}} \frac{|\xi|^\frac12\wh{g}(\xi,\tau)}{\la\tau-\alpha\xi^3\ra^b} \frac{\wh{f_1}(\xi_1,\tau_1)}{\la\tau_1\ra^\gamma} \frac{|\xi_2|^\frac12\wh{f_2}(\xi_2,\tau_2)}{\la\tau_2-\alpha\xi_2^3\ra^b} d\xi d\xi_1 d\tau d\tau_1\\
&\lsim \left\|\left[\frac{|\xi|^\frac12\wh{g}(\xi,\tau)}{\la\tau-\alpha\xi^3\ra^b}\right]^\vee\right\|_{L^4_xL^2_t}
\left\|\left[\frac{\wh{f_1}(\xi,\tau)}{\la\tau\ra^\gamma}\right]^\vee\right\|_{L^2_xL^\infty_t}
\left\|\left[\frac{|\xi|^\frac12\wh{f_2}(\xi,\tau)}{\la\tau-\alpha\xi^3\ra^b}\right]^\vee\right\|_{L^4_xL^2_t}\\
&\lsim \|g\|_{L^2_{x,t}} \|f_1\|_{L^2_{x,t}} \|f_2\|_{L^2_{x,t}}
\end{align*}
where we've used \eqref{Sobolev} and \eqref{Kato_p} with $p=4$. Note that $\frac{4-2}{2(4)} = \frac14<b$ as required.\\

\textit{Case 3} ($|\xi_2|\lsim 1\ll |\xi|$)\\
Repeat the previous argument, but with the low frequency term in $\xi_2, \tau_2$ instead of $\xi_1, \tau_1$.\\

\textit{Case 4A} ($|\xi|,|\xi_1|, |\xi_2|\gg1$, and $\xi, \xi_1 \in A'$)\\
 As in Case 2A in the proof of $(\ref{bil_1})$, we apply Cauchy-Schwarz in $\xi$ and $\tau$, so it is enough to bound

\begin{equation} \label{dual_M1**}
\sup_{\xi_1,\tau_1} \frac{\la\xi_1\ra^{-2s}}{\la\tau_1-\xi_1^3\ra^{2b}} \int \limits_{\substack{\xi=\xi_1+\xi_2 \\ \tau= \tau_1+\tau_2}} \frac{|\xi|^2\la\xi\ra^{2s}}{\la\xi_2\ra^{2s}} \frac{1}{\la\tau-\alpha\xi^3\ra^{2b} \la\tau_2-\alpha\xi_2^3\ra^{2b}}d\tau d\xi
\end{equation}
by an absolute constant. Then, since $|\xi|\lsim |\xi_1|$ on $A'$, and by Lemma \ref{A.2}
\begin{align*}
\eqref{dual_M1**}&\lsim \sup_{\xi_1,\tau_1} \int \limits_{\xi = \xi_1+\xi_2} \frac{|\xi|^2}{\la\xi_2\ra^{2s}} \frac{d\xi}{\la\tau_1-\xi_1^3\ra^{2b} \la\tau_1 - \xi_1^3 +\xi_1^3+\alpha\xi_2^3-\alpha\xi^3\ra^{4b-1}}\\
&\lsim \sup_{\xi_1,\tau_1}\int \frac{|\xi|^2}{\la\xi-\xi_1\ra^{2s}}\frac{d\xi}{ \la\xi_1(\xi-r_1^{-1}\xi_1)(\xi-r_2^{-1}\xi_1)\ra^{4b-1}} \\
&\lsim \sup_{\xi_1,\tau_1}\int \frac{|\xi|^2}{\la\xi-\xi_1\ra^{2s}}\frac{d\xi}{\la\xi^2(\xi-r_1^{-1}\xi_1)\ra^{4b-1}} \qquad \text{since }A'\subset (B')^c\\
&\lsim 1.
\end{align*}
The last line follows exactly from the estimates in Case 2A of $(\ref{bil_1})$, but with $r_1$ replaced by $r_1^{-1}$. Recall that $r_1 \neq 1$ for $0<\alpha<1$.\\

\textit{Case 4B} ($|\xi|,|\xi_1|,|\xi_2|\gg1$, and $\xi,\xi_1 \in B'$)\\
Repeat the previous subcase with the roles of $r_1$ and $r_2$ reversed. \\

\textit{Case 4C} ($|\xi|,|\xi_1|,|\xi_2| \gg 1$, and $\xi, \xi_1 \in C'$)\\
The argument here is similar to Case 2C of (\ref{bil_1}). For $\xi,\xi_1 \in C'$, we have $|\xi-r_1^{-1} \xi_1|\gsim |\xi|$ and $|\xi-r_2^{-1} \xi_1|\gsim |\xi|$. Therefore, \eqref{max_lb2} implies $M_2 \gsim |\xi_1| |\xi|^2$. \\

We separate subcases based on $M_2$. For example, if $M_2 = |\tau-\alpha \xi^3|$, then
\begin{align*}
\eqref{duality**} &\lsim \int \limits_{\substack{\xi = \xi_1+\xi_2 \\ \tau = \tau_1+\tau_2}} \frac{|\xi|\wh{g}(\xi,\tau)}{\la\tau-\alpha\xi^3\ra^b} \frac{\wh{f_1}(\xi_1,\tau_1)}{\la\tau_1-\xi_1^3\ra^b} \frac{\wh{f_2}(\xi_2,\tau_2)}{\la\tau_2-\alpha\xi_2^3\ra^b} d\xi d\xi_1 d\tau d\tau_1 \\
&\lsim \int \limits_{\substack{\xi = \xi_1+\xi_2 \\ \tau = \tau_1+\tau_2}} \frac{\la\xi\ra\wh{g}(\xi,\tau)}{\la \xi_1\ra^b\la\xi\ra^{2b}} \frac{\wh{f_1}(\xi_1,\tau_1)}{\la\tau_1-\xi_1^3\ra^b} \frac{\wh{f_2}(\xi_2,\tau_2)}{\la\tau_2-\alpha\xi_2^3\ra^b} d\xi d\xi_1 d\tau d\tau_1 \\
&\lsim \int \limits_{\substack{\xi = \xi_1+\xi_2 \\ \tau = \tau_1+\tau_2}} \wh{g}(\xi,\tau) \frac{\la\xi_1\ra^{1-3b}\wh{f_1}(\xi_1,\tau_1)}{\la\tau_1-\xi_1^3\ra^b} \frac{\la\xi_2\ra^{1-2b}\wh{f_2}(\xi_2,\tau_2)}{\la\tau_2-\alpha\xi_2^3\ra^b} d\xi d\xi_1 d\tau d\tau_1 \qquad \text{(since } \la\xi\ra\lsim \la\xi_1\ra\la\xi_2\ra)\\
&\lsim \int \limits_{\substack{\xi = \xi_1+\xi_2 \\ \tau = \tau_1+\tau_2}} \wh{g}(\xi,\tau) \frac{\wh{f_1}(\xi_1,\tau_1)}{\la\tau_1-\xi_1^3\ra^b} \frac{|\xi_2|^{1-2b}\wh{f_2}(\xi_2,\tau_2)}{\la\tau_2-\alpha\xi_2^3\ra^b} d\xi d\xi_1 d\tau d\tau_1 \qquad \text{(since } |\xi_2|\sim \la\xi_2\ra)  \\
&\lsim \|\wh{g}\|_{L^2_{\xi,\tau}}
\left\|\left[\frac{\wh{f_1}(\xi,\tau)}{\la\tau-\xi^3\ra^b}\right]^\vee\right\|_{L^4_xL^4_t}
\left\|\left[\frac{|\xi|^{1-2b}\wh{f_2}(\xi,\tau)}{\la\tau-\alpha\xi^3\ra^b}\right]^\vee\right\|_{L^4_xL^4_t} \lsim  \|g\|_{L^2_{x,t}} \|f_1\|_{L^2_{x,t}} \|f_2\|_{L^2_{x,t}} 
\end{align*}
by (\ref{Strichartz_4}), with $b\geq\frac{7}{16}$ so that $1-2b\leq \frac18$.\\

Similarly, for $M_2 = |\tau_1-\xi_1^3|$, 
\begin{align*}
\eqref{duality**} &\lsim \int \limits_{\substack{\xi = \xi_1+\xi_2 \\ \tau = \tau_1+\tau_2}} \frac{|\xi|\wh{g}(\xi,\tau)}{\la\tau-\alpha\xi^3\ra^b} \frac{\wh{f_1}(\xi_1,\tau_1)}{\la\xi_1\ra^b\la\xi\ra^{2b}} \frac{\wh{f_2}(\xi_2,\tau_2)}{\la\tau_2-\alpha\xi_2^3\ra^b} d\xi d\xi_1 d\tau d\tau_1 \\
&\lsim \int \limits_{\substack{\xi = \xi_1+\xi_2 \\ \tau = \tau_1+\tau_2}} \frac{|\xi|^{1-2b}\wh{g}(\xi,\tau)}{\la\tau-\alpha\xi^3\ra^b} \wh{f_1}(\xi_1,\tau_1) \frac{\wh{f_2}(\xi_2,\tau_2)}{\la\tau_2-\alpha\xi_2^3\ra^b} d\xi d\xi_1 d\tau d\tau_1 \lsim \|\wh{g}\|_{L^2_{\xi,\tau}}\|\wh{f_1}\|_{L^2_{\xi,\tau}}\|\wh{f_2}\|_{L^2_{\xi,\tau}}.
\end{align*}

The proof is virtually identical if $M_2 = |\tau_2-\alpha\xi_2^3|$. 
\end{proof}
\subsection{Proof of \eqref{bil_4}}
\begin{proof}
We must show
$$\| \del_x (uv) \|_{X^{\frac12+,\frac{2s-1}{6}-b}_\alpha} \lsim \|u\|_{X^{s,b} \cap V^\gamma} \|v\|_{X^{s,b}_\alpha \cap V^{\gamma}}.$$
By duality, it is enough to bound 
\begin{equation} \label{duality***}
\int \limits_{\substack{\xi = \xi_1+\xi_2 \\ \tau = \tau_1+\tau_2}} \frac{|\xi| \la\xi\ra^{\frac12+} \wh{g}(\xi,\tau)}{\la\tau-\alpha\xi^3\ra^{b-\frac{2s-1}{6}}} \frac{\wh{f_1}(\xi_1,\tau_1)}{\beta_1(\xi_1,\tau_1)}\frac{\wh{f_2}(\xi_2,\tau_2)}{\beta_\alpha(\xi_2,\tau_2)} d\xi d\xi_1 d\tau d\tau_1
\end{equation}
by $\|g\|_{L^2_{x,t}} \|f_1\|_{L^2_{x,t}} \|f_2\|_{L^2_{x,t}}$. We follow the same cases as in the proof of (\ref{bil_2}).\\ 

\textit{Case 1} ($|\xi|\lsim 1$)\\
Recall that $b>\frac{s+1}{6}>\frac{2s-1}{6}$ since $s<2$. Then
\begin{align*}
\eqref{duality***} &\lsim \int \limits_{\substack{\xi = \xi_1+\xi_2 \\ \tau = \tau_1+\tau_2}} \wh{g}(\xi,\tau) \frac{\wh{f_1}(\xi_1,\tau_1)}{\la\tau_1-\xi_1^3\ra^b} \frac{\wh{f_2}(\xi_2,\tau_2)}{\la\tau_2-\alpha \xi_2^3\ra^b} d\xi d\xi_1 d\tau d\tau_1\\
&\lsim \|\wh{g}\|_{L^2_{\xi,\tau}} \left\|\left[\frac{\wh{f_1}}{\la\tau -\xi^3\ra^b}\right]^\vee\right\|_{L^4_xL^4_t} \left\|\left[\frac{\wh{f_2}}{\la\tau -\alpha\xi^3\ra^b}\right]^\vee\right\|_{L^4_xL^4_t}\\
& \lsim \|g\|_{L^2_{x,t}} \|f_1\|_{L^2_{x,t}} \|f_2\|_{L^2_{x,t}}
\end{align*}
exactly as before.\\

\textit{Case 2} ($|\xi_1|\lsim 1\ll |\xi|\sim|\xi_2|$)\\
We can proceed as before, but with a different choice of $p$ in \eqref{Kato_p}. Recall $\frac12<s<2$ and $b> \frac{s+1}{6}>\frac{2s-1}{6}$, so that
\begin{align*}
\eqref{duality***} &\lsim \int \limits_{\substack{\xi = \xi_1+\xi_2 \\ \tau = 
	\tau_1+\tau_2}} \frac{|\xi|\la\xi\ra^{\frac12+} \wh{g}(\xi,\tau)}{\la\tau-\alpha\xi^3\ra^
	{b-\frac{2s-1}{6}}} \frac{\wh{f_1}(\xi_1,\tau_1)}{\la\tau_1\ra^\gamma}
	\frac{\wh{f_2}(\xi_2,\tau_2)}{\la\xi_2\ra^s \la\tau_2-\alpha\xi_2^3\ra^b} 
	d\xi_1 d\xi d\tau_1 d\tau \\
&\lsim \int \limits_{\substack{\xi = \xi_1+\xi_2 \\ \tau = 
	\tau_1+\tau_2}} \frac{\la\xi\ra^{2b-\frac{2s-1}{3}-} \wh{g}(\xi,\tau)}
	{\la\tau-\alpha\xi^3\ra^{b-\frac{2s-1}{6}}}\frac{\wh{f_1}(\xi_1,\tau_1)}
	{\la\tau_1\ra^\gamma}\frac{\la\xi_2\ra^{1+\frac{2s-1}{3}-2b+}\wh{f_2}
	(\xi_2,\tau_2)}{\la\tau_2-\alpha\xi_2^3\ra^b} 
	d\xi_1 d\xi d\tau_1 d\tau
\end{align*}

\begin{align*}
&\lsim \left\|\left[\frac{|\xi|^{2b-\frac{2s-1}{3}-}\wh{g}(\xi,\tau)}
	{\la\tau-\alpha\xi^3\ra^{b-\frac{2s-1}{6}}}\right]^\vee\right\|_{L^
	{\frac{3}{s+1-3b}-}_xL^2_t}\\
&\qquad \qquad \left\|\left[\frac{\wh{f_1}(\xi,\tau)}{\la\tau\ra^\gamma}
	\right]^\vee\right\|_{L^2_xL^\infty_t}\left\|\left[\frac{|\xi|^
	{\frac{2s+2}{3}-2b+}\wh{f_2}(\xi,\tau)}{\la\tau-\alpha\xi^3\ra^b}\right]^
	\vee\right\|_{L^{\frac{6}{6b-2s+1}+}_xL^2_t}\\
&\lsim \|g\|_{L^2_{x,t}} \|f_1\|_{L^2_{x,t}} \|f_2\|_{L^2_{x,t}}.
\end{align*}

\textit{Case 3} ($|\xi_2|\lsim 1\ll |\xi|\sim|\xi_1|$)\\
Repeat the above argument, but with the low frequency term in $\xi_2,\tau_2$ instead of $\xi_1,\tau_1$.\\

\textit{Case 4A} ($|\xi|,|\xi_1|,|\xi_2|\gg1$, and $\xi,\xi_1 \in A'$)\\
We proceed as in Case 4A in the proof of (\ref{bil_3}). By Cauchy-Schwarz in $\xi$ and $\tau$, it is enough to bound
\begin{equation} \label{M1***}
 \sup_{\xi_1,\tau_1} \frac{1}{\la\xi_1\ra^{2s}\la\tau_1-\xi_1^3\ra^{2b}} \int_{\xi = \xi_1+\xi_2} \frac{|\xi|^2\la\xi\ra^{1+} d\tau d\xi}{\la\xi_2\ra^{2s} \la\tau-\alpha\xi^3\ra^{2b-\frac{2s-1}{3}} \la\tau-\tau_1-\alpha\xi_2^3\ra^{2b}}
\end{equation}
by an absolute constant. Recall that $|\xi|\lsim |\xi_1|$ on $A'$. Then
\begin{align*}
\eqref{M1***}&\lsim \sup_{\xi_1,\tau_1} \frac{\la\xi_1\ra^{3+}}{\la\xi_1\ra^{2s}\la\tau_1-\xi_1^3\ra^{2b}} \int_{\xi = \xi_1+\xi_2} \frac{d\xi}{\la\xi_2\ra^{2s} \la\tau_1-\xi_1^3+\xi_1^3+\alpha\xi_2^3-\alpha\xi^3\ra^{4b-\frac{2s+2}{3}}} \\
&\lsim \sup_{\xi_1,\tau_1} \frac{\la\xi_1\ra^{3+}}{\la\xi_1\ra^{2s}\la\tau_1-\xi_1^3\ra^{\frac{2s+2}{3}-2b}} \int \frac{d\xi}{\la\xi-\xi_1\ra^{2s} \la \xi_1(\xi-r_1^{-1}\xi_1)(\xi-r_2^{-1}\xi_1)\ra^{4b-\frac{2s+2}{3}}} \\
&\lsim \sup_{\xi_1,\tau_1} \la\xi_1\ra^{3-2s+} \int \frac{d\xi}{\la\xi-\xi_1\ra^{2s} \la \xi^2(\xi-r_1^{-1}\xi_1)\ra^{4b-\frac{2s+2}{3}}}\,.
\end{align*}
This is the supremum from Case 2A of the proof of (\ref{bil_3}), but with $r_1$ replaced by $r_1^{-1}$, and we've already shown this supremum to be bounded by a constant.\\

\textit{Case 4B} ($|\xi|,|\xi_1|,|\xi_2|\gg1$, and $\xi,\xi_1 \in B'$)\\
Repeat the previous subcase with the roles of $r_1$ and $r_2$ reversed. \\

\textit{Case 4C} ($|\xi|,|\xi_1|,|\xi_2|\gg1$, and $\xi,\xi_1 \in C'$)\\
As before, we consider separate subcases depending on $M_2$. Recall that for this subcase, by (\ref{max_lb2}) we have $M_2 \gsim |\xi|^2 |\xi_1|$. \\

If $M_2 = |\tau-\alpha\xi^3|$, then
\begin{align*}
\eqref{duality***} &\lsim  \int\limits_{\substack{\xi = \xi_1+\xi_2 \\ \tau = \tau_1+\tau_2}} \frac{|\xi|\la\xi\ra^{\frac12+}\wh{g}(\xi,\tau)}{\la\xi\ra^{2b-\frac{2s-1}{3}}} \frac{\wh{f_1}(\xi_1,\tau_1)}{\la\xi_1\ra^{b-\frac{2s-1}{6}+s}\la\tau_1-\xi_1^3\ra^b} \frac{\wh{f_2}(\xi_2,\tau_2)}{\la\xi_2\ra^{s}\la\tau_2-\alpha\xi_2^3\ra^b}d\xi d\xi_1 d\tau d\tau_1 \\
&\lsim \int\limits_{\substack{\xi = \xi_1+\xi_2 \\ \tau = \tau_1+\tau_2}} \frac{\la\xi\ra^{1-2b}\wh{g}(\xi,\tau)}{\la\xi\ra^{s-\frac{2s-1}{3}-\frac12-}} \frac{\wh{f_1}(\xi_1,\tau_1)}{\la\xi_1\ra^{b-\frac{2s-1}{6}}\la\tau_1-\xi_1^3\ra^b} \frac{\wh{f_2}(\xi_2,\tau_2)}{\la\tau_2-\alpha\xi_2^3\ra^b}d\xi d\xi_1 d\tau d\tau_1 \\
&\lsim \int\limits_{\substack{\xi = \xi_1+\xi_2 \\ \tau = \tau_1+\tau_2}} \frac{\wh{g}(\xi,\tau)}{\la\xi\ra^{\frac{2s-1}{6}-}} \frac{\la\xi_1\ra^{1-2b}\wh{f_1}(\xi_1,\tau_1)}{\la\tau_1-\xi_1^3\ra^b} \frac{\la\xi_2\ra^{1-2b}\wh{f_2}(\xi_2,\tau_2)}{\la\tau_2-\alpha\xi_2^3\ra^b}d\xi d\xi_1 d\tau d\tau_1 \\
&\lsim \int\limits_{\substack{\xi = \xi_1+\xi_2 \\ \tau = \tau_1+\tau_2}} \wh{g}(\xi,\tau)\frac{|\xi_1|^{1-2b}\wh{f_1}(\xi_1,\tau_1)}{\la\tau_1-\xi_1^3\ra^b} \frac{|\xi_2|^{1-2b}\wh{f_2}(\xi_2,\tau_2)}{\la\tau_2-\alpha\xi_2^3\ra^b}d\xi d\xi_1 d\tau d\tau_1 \\
&\lsim \|g\|_{L^2_{\xi,\tau}} \|f_1\|_{L^2_{\xi,\tau}} \|f_2\|_{L^2_{\xi,\tau}}
\end{align*}
where we used $\frac12<s<2$, and we require $b\geq\frac{7}{16}$ to use (\ref{Strichartz_4}) in the last line.\\

If $M_2 = |\tau_1-\xi_1^3|$, we can apply the tradeoff argument from the proof of (\ref{bil_3}):
$$\la\tau_1-\xi_1^3\ra^b  = \la\tau_1-\xi_1^3\ra^{b-\frac{2s-1}{6}}\la\tau_1-\xi_1^3\ra^{\frac{2s-1}{6}} \geq \la\tau_1-\xi_1^3\ra^{b-\frac{2s-1}{6}}\la\tau-\alpha\xi^3\ra^{\frac{2s-1}{6}}.$$
In this subcase, we have
\begin{align*}
\eqref{duality***} &\lsim  \int\limits_{\substack{\xi = \xi_1+\xi_2 \\ \tau = \tau_1+\tau_2}} \frac{|\xi|\la\xi\ra^{\frac12+}\wh{g}(\xi,\tau)}{\la\tau-\alpha\xi^3\ra^{b-\frac{2s-1}{6}}} \frac{\wh{f_1}(\xi_1,\tau_1)}{\la\xi_1\ra^{s}\la\tau_1-\xi_1^3\ra^b} \frac{\wh{f_2}(\xi_2,\tau_2)}{\la\xi_2\ra^{s}\la\tau_2-\alpha\xi_2^3\ra^b}d\xi d\xi_1 d\tau d\tau_1 \\
&\lsim  \int\limits_{\substack{\xi = \xi_1+\xi_2 \\ \tau = \tau_1+\tau_2}} \frac{|\xi|\la\xi\ra^{\frac12+}\wh{g}(\xi,\tau)}{\la\tau-\alpha\xi^3\ra^b} \frac{\wh{f_1}(\xi_1,\tau_1)}{\la\xi_1\ra^{s}\la\tau_1-\xi_1^3\ra^{b-\frac{2s-1}{6}}} \frac{\wh{f_2}(\xi_2,\tau_2)}{\la\xi_2\ra^{s}\la\tau_2-\alpha\xi_2^3\ra^b}d\xi d\xi_1 d\tau d\tau_1\\
&\lsim  \int\limits_{\substack{\xi = \xi_1+\xi_2 \\ \tau = \tau_1+\tau_2}} \frac{|\xi|\la\xi\ra^{\frac12-s+}\wh{g}(\xi,\tau)}{\la\xi\ra^{2b-\frac{2s-1}{3}}\la\tau-\alpha\xi^3\ra^b} \frac{\wh{f_1}(\xi_1,\tau_1)}{\la\xi_1\ra^{b-\frac{2s-1}{6}}} \frac{\wh{f_2}(\xi_2,\tau_2)}{\la\tau_2-\alpha\xi_2^3\ra^b}d\xi d\xi_1 d\tau d\tau_1\\
&\lsim  \int\limits_{\substack{\xi = \xi_1+\xi_2 \\ \tau = \tau_1+\tau_2}} \frac{1}{\la\xi\ra^{\frac{2s-1}{6}-}}\frac{|\xi|^{1-2b}\wh{g}(\xi,\tau)}{\la\tau-\alpha\xi^3\ra^b} \wh{f_1}(\xi_1,\tau_1) \frac{\wh{f_2}(\xi_2,\tau_2)}{\la\tau_2-\alpha\xi_2^3\ra^b}d\xi d\xi_1 d\tau d\tau_1\\
&\lsim \|g\|_{L^2_{\xi,\tau}} \|f_1\|_{L^2_{\xi,\tau}} \|f_2\|_{L^2_{\xi,\tau}}
\end{align*}
where we've used $\frac{2s-1}{6}>0$ for $s>\frac12$ and then (\ref{Strichartz_4}) with $b>\frac{7}{16}$, just as in the previous subcase.\\

If $M_2 = |\tau_2-\alpha\xi_2^3|$, the proof is virtually identical.

\end{proof}

\section{Proof of Theorem \ref{thrm:main}}\label{sec:proof}

With the bilinear estimates in hand, we are now ready to prove our main result. 
We fix $0<\alpha<1$, $0<s<2$, $s\neq \frac12,\frac32$, $u_0, v_0 \in H^s_x(\R^+)$, $f,g \in H_t^{\frac{s+1}{3}}(\R^+)$, with the requirement that $u_0(0) = f(0)$ and $v_0(0) = g(0)$ if $s>\frac12$. Then choose extensions $\tu_0, \tv_0 \in H^s_x(\R)$ with $\|\tu_0\|_{H^s(\R)} \lsim \|u_0\|_{H^s(\R^+)}$ and $\|\tv_0\|_{H^s(\R)} \lsim \|v_0\|_{H^s(\R^+)}$, which is possible by \eqref{R+norm}.

$T<1$ will be chosen later. To choose $b$ and $\gamma$, we first choose a small $\epsilon>0$ depending on $s$. If $s<\frac12$, choose $\epsilon<\frac12-\max(\frac{3-s}{6}, \frac{7}{16})$. On the other hand, if $\frac12<s<2$, choose $\epsilon< \frac12 - \max(\frac{s+1}{6},\frac{7}{16})$ instead. Then take $\gamma = \frac12+\epsilon$ and $b = \frac12-2\epsilon$. This ensures that
$\frac12<\gamma \leq \frac12+$, $0
<\frac12-b-$, and all conditions for Proposition \ref{prop:bil} are satisfied. 

\subsection{Local Theory} \label{subsec:local} In order to show the local existence of solutions to \eqref{majdaIBVP}, we consider the operators $\Gamma_1$ and $\Gamma_2$ from the right side of the integral formulation \eqref{weaksol}:

\begin{equation*} 
\begin{split}
\Gamma_1 (u,v) &:= \eta(t)W^t \tu_0 + \eta(t)\int_0^t W^{t-t'} F(x,t') \,dt' + 2\eta\Re W_1(f-p) (x,t) \\
\Gamma_2 (u,v) &:= \eta(t)W^t_\alpha \tv_0 + \eta(t)\int_0^t W_\alpha^{t-t'} G(x,t')\, dt' + 2\eta\Re W_1(g-q) (\sqrt[3]{\alpha}\,x,t) 
\end{split}
\end{equation*}

where:\\
\\
$
\begin{array}{lll}
F(u,v) = \eta(t/T) vv_x & \qquad&  G(u,v) = \eta(t/T) (uv)_x\\
p = \eta(t) D_0(W^t \tu_0) + \eta(t) D_0 \left(\int_0^t W^{t-t'} F dt'\right) &\qquad &
q = \eta(t) D_0(W^t_\alpha \tv_0) + \eta(t) D_0 \left(\int_0^t W_\alpha^{t-t'} G dt'\right).
\end{array}
$\\

We aim to prove that the operator $\Gamma$, defined by

\begin{equation*}
\Gamma \begin{bmatrix}u \\ v\end{bmatrix} = \begin{bmatrix}
\Gamma_1(u,v) \\ \Gamma_2(u,v)
\end{bmatrix}
\end{equation*}
has a fixed point in the space 
$Y\times Y_\alpha:= (X^{s,b} \cap V^\gamma) \times (X^{s,b}_\alpha \cap V^\gamma).$ We first establish that $\Gamma$ is a bounded operator on $Y\times Y_\alpha$. Using Lemmas \ref{lin_V} and \ref{xHs_embed}, along with the fact that $s>0$,
\begin{align*}
\|\eta W^t \tu_0 \|_Y \lsim \| \eta W^t \tu_0\|_{X^{s,b}} + \|\eta W^t \tu_0\|_{V^\gamma} 
\lsim \|\tu_0\|_{H^s} + \|\tu_0\|_{L^2} \lsim \|\tu_0\|_{H^s}. 
\end{align*}
Next, by Lemmas \ref{Duham} and \ref{extractT}, followed by the bilinear estimate \eqref{bil_1}
\begin{align*}
\left\| \eta \int_0^t W^{t-t'} F(x,t') \,dt' \right\|_{X^{s,b}} &\leq \left\| \eta \int_0^t W^{t-t'} F(x,t') \,dt' \right\|_{X^{s,\frac12+}} \lsim \|F\|_{X^{s,-\frac12+}} \\
&\lsim T^{\frac12-b-} \|\del_x(v^2)\|_{X^{s,-b}} \lsim T^{\frac12-b-} \|v\|^2_{Y_\alpha}.
\end{align*}
Note that $X^{0,\gamma} \subseteq V^\gamma$. Then, if we choose $\frac12<\gamma\leq \frac12+$, we can apply Lemma \ref{Duham} to bound
\begin{align*}
\left\| \eta \int_0^t W^{t-t'} F(x,t') \,dt' \right\|_{V^\gamma} \lsim \left\| \eta \int_0^t W^{t-t'} F(x,t') \,dt' \right\|_{X^{0,\gamma}} \lsim \|F\|_{X^{0,-\frac12+}} \leq \|F\|_{X^{s,-\frac12+}}.
\end{align*}
Combining the previous two lines gives
$$\left\| \eta \int_0^t W^{t-t'} F(x,t') \,dt' \right\|_{Y} \lsim T^{\frac12-b-} \|v\|^2_{Y_\alpha}.$$
By Proposition \ref{Boundary_Xsb}, Corollary \ref{Boundary_V} (by our choice of $\gamma$,  $\frac12<\gamma\leq\frac{s}{3}$), and Lemma \ref{lem_Kato}
\begin{align*}
\| 2\Re W_1 (f-p) \|_{Y} &\lsim \|f-p\|_{H^{\frac{s+1}{3}}_t(\R^+)}\\
& \lsim \|f \|_{H^{\frac{s+1}{3}}_t(\R^+)} + \|\eta W^t \tu_0\|_{L^\infty_x H^{\frac{s+1}{3}}_t} + \left\| \eta \int_0^t W^{t-t'} F \, dt' \right\|_{L^\infty_x H^{\frac{s+1}{3}}_t}\\
&\lsim \|f \|_{H^{\frac{s+1}{3}}_t(\R^+)} + \| \tu_0\| _{H^s} + \left\| \eta \int_0^t W^{t-t'} F \, dt' \right\|_{L^\infty_x H^{\frac{s+1}{3}}_t}.
\end{align*}
We bound the last term using Proposition \ref{Duham_time_trace}. Recall that this can introduce an additional term depending on whether $\frac12<s<2$ or $s<\frac12$. Either way, we then apply Lemma \ref{extractT} followed by the appropriate bilinear estimate, \eqref{bil_1} or \eqref{bil_3} 
\begin{align*}
\left\| \eta \int_0^t W^{t-t'} F \, dt'\right\|_{L^\infty_x H^{\frac{s+1}{3}}_t} & \lsim  \|F\|_{X^{s,-\frac12+}} + \chi_{(\frac12,2)}(s) \, \|F\|_{X^{\frac12+, \frac{s-2}{3}}} \\
& \lsim T^{\frac12-b-}\|\partial_x (v^2)\|^2_{X^{s,-b}} + \chi_{(\frac12,2)}(s) \, T^{\frac12-b} \| \del_x(v^2)\|_{X^{\frac12+,\frac{2s-1}{6}-b}} \\
& \lsim  T^{\frac12-b-} \|v\|^2_{Y_\alpha}. 
\end{align*}
Therefore we have shown
$$\|\Gamma_1(u,v) \|_Y \lsim \|\tu_0\|_{H^s} + \|f\|_{H^{\frac{s+1}{3}}(\R^+)} + T^{\frac12-b-} \, \|v\|^2_{Y_\alpha} .$$

We can apply the same steps to $\Gamma_2(u,v)$ to show
$$\|\Gamma_2(u,v) \|_{Y_\alpha} \lsim \|\tv_0\|_{H^s} + \|g\|_{H^{\frac{s+1}{3}}(\R^+)} + T^{\frac12-b-} \, \|u\|_Y \, \|v\|_{Y_\alpha} .$$
Although the linear estimates from Section \ref{sec:apriori} are posed in terms of $W^t$ rather than $W_\alpha^t$, they can still be applied after a simple Fourier scaling, similar to Remark \ref{rmk:Fourier_scaling}. 

Having established that $\Gamma$ is bounded, we apply similar calculations for the differences in order to show $\Gamma$ is a contraction. For example,
\begin{align*}
\|\Gamma_2(u,v) - \Gamma_2(u',v')\|_{X^{s,b}_\alpha} &= \left\| \eta \int_0^t W^{t-t'} [G-G'] \,dt' \right\|_{X^{s,b}_\alpha}\\
&\lsim T^{\frac12-b-}\|\del_x(uv-u'v')\|_{X^{s,-b}_\alpha}\\
&\lsim T^{\frac12-b-}\|\del_x [(u-u')(v+v')]\|_{X^{s,-b}_\alpha}+T^{\frac12-b-}\|\del_x [(v-v')(u+u')]\|_{X^{s,-b}_\alpha}\\
& \lsim T^{\frac12-b-} \| v+v'\|_{Y_\alpha} \, \|u-u'\|_Y +  T^{\frac12-b-} \| u+u'\|_{Y} \, \|v-v'\|_{Y_\alpha}.
\end{align*}
The calculation for $\Gamma_1$ is similar. Therefore, by choosing $T<1$ small enough, we can ensure $\Gamma$ is a contraction mapping on a ball in $Y \times Y_\alpha$ around 0, with radius depending on the initial and boundary data $\|u_0\|_{H^s}$, $\|v_0\|_{H^s}$, $\|f\|_{H^{\frac{s+1}{3}}}$, and $\|g\|_{H^{\frac{s+1}{3}}}$. The Banach fixed point theorem then guarantees the existence of a solution $(u,v) \in Y \times Y_{\alpha}$ to \eqref{weaksol}.\\ 

We now verify that $u = \Gamma_1(u,v)$ of this fixed point lies in $C^0_tH^s_x([0,T]\times \R)$. Similar calculations apply to $v$.
By \eqref{W_unitary} and \eqref{R+norm}, we have
$$ \|\eta W^t \tu_0 \|_{C^0_tH^s_x} \lsim \| \tu_0\|_{H^s(\R)} \lsim \|u_0\|_{H^s(\R^+)}.$$
By \eqref{Xsb_embed}, and the contraction argument above, 
$$ \left\| \eta \int_0^t W^{t-t'} F(x,t') \,dt' \right\|_{C^0_tH^s_x} \lsim \left\| \eta \int_0^t W^{t-t'} F(x,t') \,dt' \right\|_{X^{s,\frac12+}} \lsim \ldots \lsim T^{\frac12-b-} \|v\|^2_{Y_\alpha}.$$
Then by Proposition 4.7 and the bounds in the contraction argument,
$$\|W_1(f-p)\|_{C^0_tH^s_x} \lsim \|f-p\|_{H^{\frac{s+1}{3}}(\R^+)} \lsim \ldots \lsim \|u_0\|_{H^s(\R^+)} + \|f\|_{H^{\frac{s+1}{3}}(\R^+)} + T^{\frac12-b-} \|v\|^2_{Y_\alpha}.$$

We must also verify that $u = \Gamma_1(u,v) \in C^0_xH^{\frac{s+1}{3}}_t (\R \times [0,T])$. 
In the contraction argument above we showed
$$\left\| \eta W^t \tu_0 + \eta \int_0^t W^{t-t'} F \, dt'\right\|_{C^0_x H^{\frac{s+1}{3}}_t} \lsim \|\tu_0\|_{H^s}  + T^{\frac12-b-} \|v\|^2_{Y_\alpha}.$$

It remains to check the boundary term. By Proposition 4.8,
$$\norm{\eta W_1 (f-p)}_{C^0_x H^{\frac{s+1}{3}}_t} \lsim \norm{f-p}_{H^{\frac{s+1}{3}}_t (\R^+)}\lsim \|\tu_0\|_{H^s}  + T^{\frac12-b-} \|v\|^2_{Y_\alpha}.$$

Similar estimates apply to $v$. Therefore our fixed point $(u,v)$ lies in the desired Banach space from Definition \ref{def:LWP}. This completes the proof of the local existence of solutions to \eqref{majdaIBVP}. 

The continuous dependence of these local solutions on the initial and boundary data also follows from the fixed point argument and the a priori estimates. The calculations are very similar to those above. For instance, if $(u,v)$ and $(u_n, v_n)$ are two solutions with different initial and boundary data, and we take $T$ to be the lesser of the two local existence times, we can easily show
\begin{align*}
\|u-u_n\|_Y + \|v-v_n\|_{Y_\alpha} \leq& C \big( \|u_0-u_{0,n}\|_{H^s(\R^+)} + \|v_0-v_{0,n}\|_{H^s(\R^+)} \big) \\
&\quad + C \big(\|f-f_n\|_{H^{\frac{s+1}{3}}(\R^+)} + \|g-g_n\|_{H^{\frac{s+1}{3}}(\R^+)} \big)\\
& \quad + C_1 T^{\frac12-b-} \|u-u_n\|_Y + C_2 T^{\frac12-b-} \|v-v_n\|_{Y_\alpha}
\end{align*}
where $C$ is an absolute constant, and $C_1$ and $C_2$ depend on the radii of the balls in the fixed point arguments (and thus on the size of the initial and boundary data). Because the local existence times were chosen to make $\Gamma$ a contraction, we have 
$$C_i T^{\frac12-b-} < 1 \quad \text{for }i=1,2$$
and continuous dependence in $Y \times Y_\alpha$ follows. The continuous dependence in $C^0_tH^s_x$ and $C^0_xH^{\frac{s+1}{3}}_t$ are proven similarly.

\subsection{Uniqueness for $s>\frac32$} \label{subsec:uniqueness}

To complete the proof of Theorem \ref{thrm:main}. We consider two solutions $(u,v)$ and $(u',v')$ of the IBVP \eqref{majdaIBVP}. By the local theory in the previous section, these solutions are in $C_tH^s_x([0,T]\times \R^+)$.  For $s>\frac32$, uniqueness follows from the energy estimates described in the following lemma together with Gronwall's inequality. 

\begin{lemma} \label{energy_bound}
Let $s>\frac32$ and let $(u,v)$ and $(u',v')$ be two solutions of \eqref{majdaIBVP}. Define
$$I:= \|u-u'\|^2_{L^2(\R^+)} + \|v-v'\|^2_{L^2(\R^+)}$$
Then, 
$$\del_t I \lsim (\max_{f = u,u',v,'v} \|f\|_{H^s(\R^+)})\,  I.$$

\end{lemma}

\begin{proof}
For this proof we write $L^2$ to mean $L^2(\R^+)$. First suppose $(u,v)$ and $(u',v')$ are smooth solutions of \eqref{majdaIBVP}. Then we compute
\begin{align*}
\del_t \|u-u'\|^2_{L^2} &= -2 \int_0^\infty (u-u')(u-u')_{xxx} \,dx -2 \int_0^\infty (u-u')(vv_x-v'v'_x) \,dx \\
&= -\int_0^\infty (u-u')[(v-v')(v+v')]_x \,dx\\
& = -\int_0^\infty (u-u')(v-v') (v+v')_x \,dx - \int_0^\infty (u-u')(v+v')(v-v')_x \,dx \\
&:= I_1 + I_2
\end{align*}
The first integral in the first line above is seen to be zero after an integration by parts (note that $u(0,t) = u'(0,t) = f(t)$ for all $t>0$, so the boundary terms vanish). 

Likewise 
\begin{align*}
\del_t \|v-v'\|^2_{L^2} &= -2\alpha \int_0^\infty (v-v')(v-v')_{xxx}\,dx -2 \int_0^\infty (v-v')[(uv)_x-(u'v')_x] \, dx \\
& = -\int_0^\infty (v-v')[(u-u')(v+v')]_x dx -\int_0^\infty (v-v')[(u+u')(v-v')]_x \, dx
\end{align*}
where the first integral is again zero and we've used the identity
$$2(uv-u'v') = (u-u')(v+v') + (u+u')(v-v').$$ 
Expanding gives 
\begin{align*}
\del_t \|v-v'\|^2_{L^2} &= -\int_0^\infty(v-v')(u-u')(v+v')_x \, dx - \int_0^\infty (v-v')(v+v')(u-u')_x \, dx \\
& \qquad \qquad -\int_0^\infty (v-v')^2 (u+u')_x \, dx - \int_0^\infty (v-v')(u+u')(v-v')_x \, dx\\
&:= I_3 + I_4 + I_5 + I_6.
\end{align*}

By Cauchy-Schwarz we have 
\begin{align*}
|I_1| &\leq \|u-u'\|_{L^2} \|v-v'\|_{L^2} \|(v+v')_x\|_{L^\infty}\\
|I_3| &\leq \|u-u'\|_{L^2} \|v-v'\|_{L^2} \|(v+v')_x\|_{L^\infty}\\
|I_5| &\leq \|v-v'\|^2_{L^2} \|(u+u')_x\|_{L^\infty}.
\end{align*}

For $s>\frac32$, the Sobolev embedding (along with the elementary identity $2ab \leq a^2 + b^2$), gives the desired bound for the above terms. 

For $I_6$, we simply integrate by parts
$$|I_6| = \left| \int_0^\infty \frac12 (v-v')^2 (u+u')_x \, dx \right| \leq \frac12 \|v-v'\|^2_{L^2} \|(u+u')_x \|_{L^\infty}.$$ 

Finally we treat $I_2$ and $I_4$ together and integrate by parts again
$$|I_2+I_4| = \left| -\int_0^\infty (v+v') [(u-u')(v-v')]_x \, dx \right| \leq \|u-u'\|_{L^2} \|v-v'\|_{L^2} \|(v+v')_x\|_{L^\infty}.$$

In fact, this argument extends to the case where $(u,v)$ and $(u',v')$ are not smooth. The idea is to take the convolution of $u-u'$ and $v-v'$ with smooth approximate identities and apply a limiting argument. A similar version of this is carried out in \cite{holmerNLS} for the NLS equation. Finally, we remark that this energy estimate approach establishes uniqueness independent of the choice of extension of the initial data as the norms in the lemma are taken on $\R^+$. 
\end{proof}

\renewcommand\thesection{A}
\section{Appendix}
In this section we record a few of the standard inequalities which are frequently useful in the proofs of a priori or multilinear estimates. 
\begin{lemma} \cite [Lemma 6.1]{chous} \label{A.1}
For $-\frac12 \leq s \leq \frac12$, we have
$$\norm{fg}_{H^s} \lsim \norm{f}_{H^{1/2+}} \norm{g}_{H^s}.$$ 
\end{lemma}

The following calculus lemmas appear in \cite{tz_cubicnls} as Lemmas A.2 and A.3, respectively.
\begin{lemma} \label{A.2}
If $0\leq \gamma \leq \beta$ and $1<\beta+\gamma$, then
$$\int \frac{1}{\jbrac{x-a_1}^\beta \jbrac{x-a_2}^\gamma} dx \lsim \jbrac{a_1-a_2}^{-\gamma} \phi_\beta (a_1-a_2),$$
\indent where
$$\phi_\beta(a) \sim
\begin{cases}
1 & \beta > 1\\
\log{(1+\jbrac{a}} & \beta = 1 \\
\jbrac{a}^{1-\beta} & \beta < 1.
\end{cases}$$
\end{lemma}

\begin{lemma} \label{A.3}

For fixed $\frac12<\rho<1$, we have

$$\int \frac{1}{\jbrac{x}^\rho \sqrt{|x-a|}} dx \lsim \frac{1}{\jbrac{a}^{\rho-\frac{1}{2}}}.$$
\end{lemma}

For completeness, we include Young's inequality as well.
\begin{lemma} [Young's Inequality] \cite[Proposition 8.9]{folland} \label{A.4}\\
If $1+\frac1r = \frac1p +\frac1q $ with $1\leq p,q,r \leq \infty$, then
$$\|f * g \|_{L^r} \leq \|f\|_{L^p} \, \|g\|_{L^q}.$$
\end{lemma}

\section*{Acknowledgements}
The author would like to thank his Ph.D advisor Nikolaos Tzirakis for many helpful discussions and comments.

\end{document}